\documentclass[10pt,reqno]{amsart}
\usepackage{latexsym,amsmath,amssymb,amscd}
\usepackage[all]{xy}

\begin{document}


\makeatletter
\@addtoreset{figure}{section}
\def\thefigure{\thesection.\@arabic\c@figure}
\def\fps@figure{h,t}
\@addtoreset{table}{bsection}

\def\thetable{\thesection.\@arabic\c@table}
\def\fps@table{h, t}
\@addtoreset{equation}{section}
\def\theequation{
\arabic{equation}}
\makeatother

\newcommand{\bfi}{\bfseries\itshape}

\newtheorem{theorem}{Theorem}
\newtheorem{acknowledgment}[theorem]{Acknowledgment}
\newtheorem{corollary}[theorem]{Corollary}
\newtheorem{definition}[theorem]{Definition}
\newtheorem{example}[theorem]{Example}
\newtheorem{lemma}[theorem]{Lemma}
\newtheorem{notation}[theorem]{Notation}
\newtheorem{problem}[theorem]{Problem}
\newtheorem{proposition}[theorem]{Proposition}
\newtheorem{remark}[theorem]{Remark}
\newtheorem{setting}[theorem]{Setting}

\numberwithin{theorem}{section}
\numberwithin{equation}{section}

\newcommand{\Bg}{\mathfrak{B}}

\renewcommand{\1}{{\bf 1}}
\newcommand{\Ad}{{\rm Ad}}
\newcommand{\Aut}{{\rm Aut}\,}
\newcommand{\ad}{{\rm ad}}
\newcommand{\botimes}{\bar{\otimes}}
\newcommand{\Ci}{{\mathcal C}^\infty}
\newcommand{\Der}{{\rm Der}\,}
\newcommand{\de}{{\rm d}}
\newcommand{\ee}{{\rm e}}
\newcommand{\End}{{\rm End}\,}
\newcommand{\ev}{{\rm ev}}
\newcommand{\hotimes}{\widehat{\otimes}}
\newcommand{\id}{{\rm id}}
\newcommand{\ie}{{\rm i}}
\newcommand{\Img}{{\rm Im}\,}
\newcommand{\Ker}{{\rm Ker}\,}
\newcommand{\Kern}{{\rm Kern}}
\newcommand{\Lie}{\text{\bf L}}
\newcommand{\Ran}{{\rm Ran}\,}
\newcommand{\RUCb}{{{\mathcal R}{\mathcal U}{\mathcal C}_b}}
\renewcommand{\Re}{{\rm Re}\,}
\newcommand{\spa}{{\rm span}\,}
\newcommand{\supp}{{\rm supp}\,}

\newcommand{\CC}{{\mathbb C}}
\newcommand{\HH}{{\mathbb H}}
\newcommand{\RR}{{\mathbb R}}
\newcommand{\TT}{{\mathbb T}}

\newcommand{\Ac}{{\mathcal A}}
\newcommand{\Bc}{{\mathcal B}}
\newcommand{\Cc}{{\mathcal C}}
\newcommand{\Dc}{{\mathcal D}}
\newcommand{\Ec}{{\mathcal E}}
\newcommand{\Fc}{{\mathcal F}}
\newcommand{\Hc}{{\mathcal H}}
\newcommand{\Jc}{{\mathcal J}}
\newcommand{\Kc}{{\mathcal K}}
\newcommand{\Rc}{{\mathcal R}}
\newcommand{\Yc}{{\mathcal Y}}
\newcommand{\Xc}{{\mathcal Xc}}

\renewcommand{\gg}{{\mathfrak g}}

\newcommand{\ZZ}{\mathbb Z}
\newcommand{\NN}{\mathbb N}
\newcommand{\BB}{\mathbb B}

\newcommand{\ep}{\varepsilon}

\newcommand{\hake}[1]{\langle #1 \rangle }

\newcommand{\scalar}[2]{\langle #1 ,#2 \rangle }
\newcommand{\vect}[2]{(#1_1 ,\ldots ,#1_{#2})}
\newcommand{\norm}[1]{\Vert #1 \Vert }
\newcommand{\normrum}[2]{{\norm {#1}}_{#2}}

\newcommand{\upp}[1]{^{(#1)}}
\newcommand{\p}{\partial}

\newcommand{\opn}{\operatorname}
\newcommand{\slim}{\operatornamewithlimits{s-lim\,}}
\newcommand{\sgn}{\operatorname{sgn}}

\newcommand{\seq}[2]{#1_1 ,\dots ,#1_{#2} }
\newcommand{\loc}{_{\opn{loc}}}

\makeatletter
\title[Inverse-closed algebras of integral operators]{Inverse-closed algebras of integral operators\\ on 
locally compact groups}
\author{Ingrid Belti\c t\u a and Daniel Belti\c t\u a}
\address{Institute of Mathematics ``Simion Stoilow'' 
of the Romanian Academy, 
P.O. Box 1-764, Bucharest, Romania}
\email{ingrid.beltita@gmail.com, Ingrid.Beltita@imar.ro}
\email{beltita@gmail.com, Daniel.Beltita@imar.ro}
\date{\today}
\makeatother

\begin{abstract} 
We construct some inverse-closed algebras of bounded integral operators 
with operator-valued kernels, 
acting in spaces of vector-valued functions on locally compact groups. 
To this end we make use of covariance algebras 
associated to $C^*$-dynamical systems defined by the $C^*$-algebras 
of right uniformly continuous functions with respect to the left regular representation. 
\end{abstract}

\maketitle


\section{Introduction}

The spectral investigation on differential operators by using the $C^*$-algebras generated by their resolvents 
has attracted much interest recently, motivated to a large extent by problems that come 
from the quantum physics; 
see for instance \cite{GI02}, \cite{BG08}, \cite{Ge11}, \cite{BG13} and the references therein.  
These operators often act in the Hilbert space $L^2(\RR^n)$, 
and yet it turned out in \cite{Ge11} that a deep insight into the spectral theory can be gained 
by working in a more general setting obtained by replacing the abelian group $(\RR^n,+)$ by 
a Lie group or even by a locally compact group $G$. 
We contribute to this circle of ideas by 
constructing some inverse-closed Banach algebras of $L^2$-bounded integral operators 
on locally compact groups, a property called sometimes the Wiener property. 
The  relevance to the spectral problems is enhanced by the fact 
that our algebras of integral operators are Banach algebras which contain the compact operators and 
are continuously and densely embedded 
into the $C^*$-algebras generated by the resolvents of differential operators 
studied in \cite{DG04} and \cite{Ge11} (Theorem~\ref{final} and Corollary~\ref{last}). 

The method of our study is provided by the link 
to the duality theory for crossed products (Remark~\ref{history}) 
and by the systematic use of covariance algebras 
associated to $C^*$-dynamical systems 
defined by the $C^*$-algebras of right uniformly continuous functions 
with respect to the left regular representation. 
This method allows us to partly simplify 
the proofs of some results from the earlier literature 
and also to obtain some new results. 

A particular feature of our work is that we deal with operator-valued integral kernels, 
which define integral operators on spaces of vector-valued functions on locally compact groups. 
This possibility was mentioned in \cite{GI02} in the case of abelian groups.  
Besides that, in the case of discrete groups, this allows us to construct inverse-closed algebras 
of bounded operators defined by block matrices. 

In this connection we recall that pseudo-differential operators with operator-valued symbols 
have been used in the study of periodic Schr\"odinger operators; 
see for instance \cite{GMS91}, \cite{RR06}, and the references therein. 
Note that in this case it is important to have specific information on some integral kernels 
related to the inverses of the operators involved. 
In abstract terms, this amounts to inverse-closedness of a Banach or even Fr\'echet algebra 
of the type considered in Corollary~\ref{last} below. 

The paper is organized as follows. 
Section~\ref{Sect1} is devoted to some preliminaries on symmetric involutive Banach algebras 
and covariance algebras of $C^*$-dynamical systems.
In Section~\ref{Sect2} we introduce the main classes of operator-valued integral kernels 
and we investigate the relationship between these kernels and the covariance algebras 
of certain $C^*$-dynamical systems. 
In this section we also obtain our main inverse-closed algebras of integral operators 
on vector-valued functions on locally compact groups 
(Theorem~\ref{main}). 
In Section~\ref{Sect3} we provide a method for constructing 
larger inverse-closed algebras of integral operators, 
and the corresponding result (Theorem~\ref{kurbatov}) applies 
in the situation that we encountered in our earlier paper \cite{BB12} 
in connection with Weyl-Pedersen calculus for unitary irreducible representations of nilpotent Lie groups. 
Finally, certain symmetry groups of our inverse-closed algebras of integral operators are studied 
in Section~\ref{Sect5}, and some of their special features are established 
in the Lie group setting, with motivation coming fr
om the recent results of \cite{BG13} and \cite{BB13}. 
The proof of Corollary~\ref{leptin2} is given in Appendix~A. 

\subsection*{General notation}
For any topological spaces $X$ and $Y$ we denote by $\Cc(X,Y)$ the set of all continuous maps from $X$ into $Y$. 
If moreover $X$ and $Y$ are smooth (maybe infinite-dimensional) manifolds, 
then $\Ci(X,Y)$ stands for the subset of $\Cc(X,Y)$ 
consisting of smooth maps. 

For any set $S$ and any Banach space $\Yc$ with the norm $\Vert\cdot\Vert_{\Yc}$, 
we denote by $\ell^\infty(S,\Yc)$ the Banach space 
consisting of all bounded functions $\phi\colon S\to\Yc$ with the norm $\Vert\phi\Vert_\infty:=\sup\limits_S\Vert\phi(\cdot)\Vert_{\Yc}$. 
If $\Yc=\CC$ then we denote simply $\ell^\infty(S):=\ell^\infty(S,\CC)$. 

If $\Ac$ is any associative complex algebra and we define $\Ac_1:=\Ac$ if there exists an unit element $\1\in\Ac$ 
and $\Ac_1:=\CC\1\dotplus \Ac$ (the unitized algebra of $\Ac$) otherwise, 
then for every element $a\in\Ac$ we define its spectrum as the set of all $z\in\CC$ for which 
$z\1-a$ is invertible in $\Ac_1$. 

For any involutive Banach algebra $\Ac$ we denote by $C^*(\Ac)$ its enveloping $C^*$-algebra 
as in \cite[Def. 10.1.10]{Pl01} (or \cite[Def. 10.4]{FD88}).

\section{Preliminaries on symmetric involutive Banach algebras}\label{Sect1}

We begin by recalling the basic notions of symmetric algebra and locally compact group. 
See for instance \cite{Bi10} for a self-contained account of various characterizations of the symmetric algebras in the setting of locally convex algebras with continuous inversion. 

\begin{definition}\label{sym}
\normalfont
An involutive Banach algebra $\Bc$ (with continuous involution) is \emph{symmetric} if for every $b\in\Bc$ the spectrum of $b^*b$ is contained in $[0,\infty)$.

A locally compact group $G$ is said to be \emph{rigidly symmetric} 
 if for every $C^*$-algebra $\Ac$ the projective tensor product $L^1(G)\hotimes\Ac$ is a symmetric Banach algebra. 
If merely $L^1(G)$ is assumed to be a symmetric Banach algebra, then we say that $G$ is a \emph{symmetric group}. 
\end{definition}

It is still unknown whether every symmetric group is rigidly symmetric. 
By \cite[Cor.~6]{Po92}, nilpotent locally compact groups are rigidly symmetric. 
Moreover, by \cite[Th.~1]{LP79}, also compact groups are rigidly symmetric. 

\begin{definition}
\normalfont
For any unital complex algebra $\Bc$ we denote by $\Bc^\times$ the group of invertible elements in $\Bc$ 
and by $\1$ its unit element, 
unless $\Bc$ is realized as an unital algebra of operators on some Hilbert space $\Hc$, 
when we denote by $\id_{\Hc}$ the identity operator. 
A unital subalgebra $\Ac$  of $\Bc$ is \emph{inverse-closed} if and only if $\Ac^\times=\Ac\cap\Bc^\times$. 
Note that we always have $\Ac^\times\subseteq\Ac\cap\Bc^\times$. 
\end{definition}

\begin{proposition}\label{univ}
If $\Ac$ is any involutive symmetric unital Banach algebra  
with its canonical homomorphism $\rho_0\colon\Ac\to C^*(\Ac)$ into its universal $C^*$-algebra, 
then $\rho_0(\Ac)$ is an inverse-closed subalgebra of $C^*(\Ac)$. 
\end{proposition}

\begin{proof}
See for instance \cite[Prop.~7.5]{Bi10} for a more general statement. 
\end{proof}

\begin{corollary}\label{reg}
Let $G$ be any locally compact group with the left regular representation 
$\lambda\colon L^1(G)\to \Bc(L^2(G))$. 
If $G$ is amenable and symmetric, 
then the unitization of the algebra of convolution operators 
$\CC\1+\lambda(L^1(G))$ is an inverse-closed subalgebra of $\Bc(L^2(G))$. 
\end{corollary}

\begin{proof}
Since $G$ is amenable, it follows by \cite[Th. 4.21]{Pa88} or \cite[Th. A.18]{Wi07} that 
the enveloping $C^*$-algebra of $L^1(G)$ is the closure of $\lambda(L^1(G))$ 
in the norm operator topology of $\Bc(L^2(G))$, which is just the reduced group $C^*$-algebra of $G$. 
That is, $C^*(L^1(G))=C^*_r(G)$
and $\lambda\colon L^1(G)\to C^*_r(G)$ is the corresponding canonical homomorphism. 
Then Proposition~\ref{univ} implies that $\CC\1+\lambda(L^1(G))$ 
is an inverse-closed subalgebra of the $C^*$-algebra $\CC\1+C^*_r(G)$ 
which is in turn inverse-closed in $\Bc(L^2(G))$ 
since every unital $C^*$-algebra is inverse-closed in any larger $C^*$-algebra. 
Consequently $\CC\1+\lambda(L^1(G))$ is an inverse-closed subalgebra of $\Bc(L^2(G))$, 
and this completes the proof. 
\end{proof}

It is noteworthy that if the group $G$ is assumed to be rigidly symmetric, 
rather than merely symmetric, in Corollary~\ref{reg}, then the conclusion holds true for the regular representation in spaces 
of vector-valued functions $L^2(G,\Hc_0)$, where $\Hc_0$ is an arbitrary Hilbert space.  

Another consequence of Proposition~\ref{univ} is Lemma~\ref{biller} below, 
which is needed in the proof of Theorem~\ref{main}. 

\begin{definition}\label{cov}
\normalfont
A \emph{$C^*$-dynamical system} $(\Ac,G,\alpha)$ consists of a $C^*$-algebra~$\Ac$ endowed with a continuous action of a locally compact group $G$ by automorphisms of~$\Ac$, 
$$\alpha\colon G\times\Ac\to\Ac,\quad (x,a)\mapsto\alpha_x(a).$$
The corresponding \emph{covariance algebra} $L^1(G,\Ac;\alpha)$ is the involutive associative Banach algebra obtained from the space of equivalence classes of Bochner integrable $\Ac$-valued functions on $G$ with the multiplication defined by 
\begin{equation}\label{cov_eq1}
(f\star h)(x)=\int\limits_G f(y)\alpha_y(h(y^{-1}x))\de y 
\end{equation}
and the involution 
\begin{equation}\label{cov_eq2}
f^*(x)=\Delta(x^{-1})\alpha_{x}(f(x^{-1})^*)
\end{equation}
for $f,h\in L^1(G,\Ac;\alpha)$ and almost every $x\in G$. 
Here $\Delta\colon G\to(0,\infty)$ is the modular function of~$G$. 
\end{definition}

\begin{definition}
\normalfont
Let $(\Ac,G,\alpha)$ be a $C^*$-dynamical system and  
$\pi\colon\Ac\to\Bc(\Hc_0)$ be a faithful $*$-representation of~$\Ac$. 
The \emph{$\pi$-regular representation} of the covariance algebra $L^1(G,\Ac;\alpha)$ is the continuous $*$-representation 
$\Pi\colon L^1(G,\Ac;\alpha)\to\Bc(L^2(G,\Hc_0))$ defined by 
$$(\Pi(f)\xi)(x)=\int\limits_G \pi(\alpha_{x^{-1}}(f(y)))\xi(y^{-1}x)\de y $$
for $f\in L^1(G,\Ac;\alpha)$, $\xi\in L^2(G,\Hc_0)$, and almost every $x\in G$. 
\end{definition}

\begin{proposition}\label{leptin1}
Let $G$ be a discrete group $(\Ac,G,\alpha)$ be a $C^*$-dynamical system. 
Then there exist a $C^*$-dynamical system $(\bar\Ac,G,\bar\alpha)$ 
with trivial action $\bar\alpha$ and an isometric $*$-homomorphism 
$\theta\colon \ell^1(G,\Ac;\alpha)\to \ell^1(G,\bar\Ac;\bar\alpha)=\ell^1(G)\hotimes\bar\Ac$. 
\end{proposition}

\begin{proof}
We may assume without loss of generality that there exist 
a complex Hilbert space $\Hc_0$ and a continuous unitary representation $ V\colon G\to\Bc(\Hc_0)$ with 
$\Ac\subseteq\Bc(\Hc_0)$ and $\alpha_x(a)= V(x)a V(x)^{-1}$ for every $a\in\Ac$ and $x\in G$ 
(see for instance \cite[7.7.1]{Pe79} or \cite[Ex.~2.14]{Wi07}). 
Let $\bar\Ac$ be the $C^*$-algebra generated by $\Ac\cup V(G)$ 
with the trivial action $\bar\alpha$ of $G$, and define 
$$\theta \colon \ell^1(G,\Ac;\alpha)\to \ell^1(G,\bar\Ac;\bar\alpha),\quad 
f\mapsto f(\cdot) V(\cdot).$$
Since $ V(x)\in\Bc(\Hc_0)$ is a unitary operator for every $x\in G$, 
it follows at once that the mapping $\theta$ is an isometry. 
Moreover, for $f,h\in \ell^1(G,\Ac;\alpha)$ we have 
$$\begin{aligned}
(\theta(f)\star\theta(h))(x)
&=\sum\limits_{y\in G}(\theta(f))(y)(\theta(h))(y^{-1}x)\\
&=\sum\limits_{y\in G} f(y) V(y)h(y^{-1}x) V(y^{-1}x)\\
&=\big(\sum\limits_{y\in G} f(y) V(y)h(y^{-1}x) V(y)^{-1} \big) V(x)\\
&=\big(\sum\limits_{y\in G} f(y)\alpha_y(h(y^{-1}x))\big) V(x) \\
&=(f\star h)(x) V(x) \\
&=(\theta(f\star h))(x).
\end{aligned}$$
Also, 
$$\begin{aligned}
\theta(f)^*(x)& =\Delta(x^{-1})((\theta(f))(x^{-1}))^*
=\Delta(x^{-1})(f(x^{-1}) V(x^{-1}))^* \\
&=\Delta(x^{-1}) V(x)f(x^{-1})^*
=\Delta(x^{-1})\alpha_x(f(x^{-1})^*) V(x)
=(\theta(f^*))(x),
\end{aligned}$$
and this concludes the proof. 
\end{proof}

The result of the above Proposition~\ref{leptin1} is actually a by-product of the proof of 
\cite[Satz 6]{Le68}.  
Note that the isometric $*$-homomorphism $\theta$ in the statement need not be surjective. 

\begin{corollary}\label{leptin2}
Let $G$ be any locally compact group with its underlying discrete group denoted by~$G_d$. 
Let $(\Ac,G,\alpha)$ be any $C^*$-dynamical system. 
If the group $G_d$ is rigidly symmetric, 
then the covariance algebra $L^1(G,\Ac;\alpha)$ is a symmetric involutive Banach algebra. 
\end{corollary}
The proof  of Corollary~\ref{leptin2} is given in Appendix~A, since it requires extra notions and notation that 
are not in the main line of the paper.

\section{Integral operators on vector-valued functions}\label{Sect2}

\begin{setting}\label{k0}
\normalfont
Throughout what follows in the present paper, 
unless otherwise mentioned, 
we shall work in a setting involving the following basic ingredients: 
\begin{enumerate}
\item $G$ stands for a unimodular locally compact topological group with a fixed Haar measure denoted by $\de x$.
\item $\Dc_0$ is a unital $C^*$-algebra of operators on some complex Hilbert space $\Hc_0$. 
\end{enumerate}
\end{setting}

\begin{notation}\label{k0.5}
\normalfont
For every complex Banach space $\Yc$ and every $p\in[1,\infty)$ 
we denote by $L^p(G,\Yc)$ the Banach space consisting of the equivalence classes of $\Yc$-valued, Bochner $p$-integrable functions on $G$ 
(see for instance \cite{Ha53}, \cite{EH53}, and \cite{El58}). 
If $\Yc=\CC$, we denote simply $L^p(G,\CC)=L^p(G)$, as usual. 
\end{notation}

\subsection*{Algebras of operator-valued integral kernels}
In the following definition we introduce several objects of major importance 
for the subsequent developments in the present paper, 
namely some spaces of operator-valued integral kernels 
on locally compact groups, operations on them along with natural norms, 
and also the corresponding integral operators. 
The basic properties of these objects which will be needed below 
(the fact that we obtain a normed algebra of integral kernels, 
or that the corresponding integral operators are bounded etc.) 
are contained in Proposition~\ref{mult} below. 

\begin{definition}\label{k1}
\normalfont 
Pick an arbitrary linear subspace $\Fc\subseteq L^1(G)$. 
If $K\colon G\times G\to \Dc_0$ is a Bochner measurable function, 
then we define $\Vert K\Vert_{\Kern_{\Fc}( G,\Dc_0)}$ as the infimum of the norms $\Vert\beta\Vert_{L^1( G)}$ 
for $\beta\in \Fc$ such that $\Vert K(x,y)\Vert\le\vert\beta(xy^{-1})\vert$ for a.e. $x,y\in G$. 
If no function $\beta$ satisfies these conditions, 
then we set $\Vert K\Vert_{\Kern_{\Fc}( G,\Dc_0)}=\infty$. 
We introduce the space of Bochner measurable functions 
$$\Kern_{\Fc}( G,\Dc_0):=\{K\colon G\times G\to \Dc_0\mid
\Vert K\Vert_{\Kern_{\Fc}( G,\Dc_0)}<\infty\}.$$
If $\Fc=L^1(G)$, then we will omit $\Fc$ from the notation $\Kern_{\Fc}( G,\Dc_0)$ 
and we also introduce the linear mapping
\begin{equation}\label{intop}
\Kern( G,\Dc_0)\to\Bc(L^2( G,\Hc_0)),\quad K\mapsto T_K,
\end{equation}
where $T_K$ is the operator on $L^2( G,\Hc_0)$ defined by the integral kernel~$K$, that is, 
$$(T_Kf)(x)=\int\limits_G K(x,y)f(y)\de y  $$
for every $f\in L^2( G,\Hc_0)$. 
We denote by $\star$ both the usual composition of integral kernels, that is, 
$$(K_1\star K_2)(x,z)=\int\limits_ G K_1(x,y)K_2(y,z)\de y\text{ for }x,z\in G,$$
and the convolution operation 
$$(\beta_1\star\beta_2)(x)=\int\limits_G\beta_1(xy^{-1})\beta_2(y)\de y\text{ for }x\in G $$ 
for $\beta_1,\beta_2\in L^1( G)$. 
Moreover, we denote $K^*(x,y):=K(y,x)^*$ for any integral kernel 
$K\colon  G\times G\to \Dc_0$. 
\end{definition}

We now record a few basic properties of the objects introduced in Definition~\ref{k1}. 
The mapping~\eqref{intop} will be called sometimes the \emph{canonical representation} of 
the associative algebra $\Kern( G,\Dc_0)$, 
although it implicitly depends on the realization of the $C^*$-algebra $\Dc_0$ as an operator algebra on $\Hc_0$.

\begin{proposition}\label{mult}
The space of $\Dc_0$-valued integral kernels 
$\Kern( G,\Dc_0)$ equipped with the above defined product and involution 
is an involutive associative Banach algebra 
with the faithful contractive $*$-representation by integral operators given by  
the mapping~\eqref{intop}.  
\end{proposition}

\begin{proof}
It easily follows by \cite[Th. 3.8]{GW04} (see also \cite{BC08a}) that 
$$ (\forall K\in \Kern( G,\Dc_0))\quad \Vert T_K\Vert\le\Vert K\Vert_{\Kern( G,\Dc_0)}.$$ 
Moreover, it is straightforward to check that $\Kern( G,\Dc_0)$ is a an associative normed $*$-algebra 
with a faithful $*$-representation defined by the mapping~\eqref{intop}, 
which is a contractive representation by the above norm estimate. 

It remains to check that the norm of $\Kern( G,\Dc_0)$ is complete. 
To this end denote by $L^{\infty,1}(G\times G,\Dc_0)$ the space of (equivalence classes of) 
Bochner measurable functions $\varphi\colon G\times G\to \Dc_0$ 
for which 
$\Vert\varphi\Vert_{\infty,1}:=\int_G\Vert\varphi(\cdot,s)\Vert_\infty\de s<\infty$. 
Then $L^{\infty,1}(G\times G,\Dc_0)$ is a Banach space by \cite[Th. 3.1]{El58} 
(see also \cite[Th. 3.1]{EH53} and \cite{Ha53}).  
Moreover, for every Bochner measurable function $\varphi\colon G\times G\to \Dc_0$ 
we have (see for instance \cite[Prop. 2.2]{Ku01})
$$\int_G\Vert\varphi(\cdot,s)\Vert_\infty\de s=\inf\{\Vert\beta\Vert_{L^1(G)}\mid \beta\in L^1(G)\text{ and }
\Vert \varphi(t,s)\Vert\le\vert\beta(s)\vert \text{ a.e. on }G\times G\}. $$
Therefore, if we define the homeomorphism 
$$\Psi\colon G\times G\to G\times G,\quad \Psi(x,y)=(y,xy^{-1})$$
with the inverse given by $\Psi^{-1}(t,s)=(st,t)$ for all $s,t\in G\times G$, 
then we obtain a surjective isometry 
$$\Kern( G,\Dc_0)\to L^{\infty,1}(G\times G,\Dc_0),\quad K\mapsto K\circ\Psi^{-1}$$
hence $\Kern( G,\Dc_0)$ is in turn a Banach space. 
\end{proof}

\subsection*{$C^*$-dynamical systems of uniformly continuous functions}
The next remark introduces the maximal space of bounded $\Dc_0$-valued functions on $G$ 
which is continuously acted on by left translations of $G$ 
and thus gives rise to a $C^*$-dynamical system. 

\begin{remark}\label{syst}
\normalfont
Let $\RUCb(G,\Dc_0)$ be the unital $C^*$-algebra consisting of 
the right uniformly continuous bounded $\Dc_0$-valued functions on $ G$. 
That is, if we define $(\alpha_xf)(y)=f(x^{-1}y)$ for arbitrary $x,y\in G$ 
and any continuous function $f\colon G\to \Dc_0$, then 
we have $f\in\RUCb(G,\Dc_0)$ if and only if 
$\Vert f\Vert_\infty:=\sup\{\Vert f(x)\Vert\mid x\in G\}<\infty$ 
and $\lim\limits_{x\to\1}\Vert\alpha_x f-f\Vert_\infty=0$. 

Then it is clear that $(\RUCb(G,\Dc_0), G,\alpha)$ is a $C^*$-dynamical system. 
For any $f\in L^1(G,\RUCb(G,\Dc_0);\alpha)$ and $x\in G$ we have 
$f(x)\in\RUCb( G,\Dc_0)$, hence we can define 
$f(x,y):=(f(x))(y)\in \Dc_0$ for almost every $y\in G$. 
This convention is also used below in other spaces, 
for instance in equation~\eqref{syst_eq1}. 

There exists a faithful $*$-representation 
$$\pi\colon\RUCb( G,\Dc_0)\to\Bc(L^2( G,\Hc_0)),\quad 
(\pi(f)\xi)(\cdot)=f(\cdot)\xi(\cdot)$$
and the corresponding $\pi$-regular representation 
\begin{equation}\label{syst_eq0}
\Pi\colon L^1(G,\RUCb(G,\Dc_0);\alpha)\to\Bc(L^2(G,L^2( G,\Hc_0))) 
\end{equation}
can be described by the formula 
\begin{equation}\label{syst_eq1}
(\Pi(f)\xi)(x,z)
=\int\limits_ G 
\underbrace{f(y,xz)}_{\hskip15pt\in \Dc_0}
\underbrace{\xi(y^{-1}x,z)}_{\hskip20pt\in\Hc_0}\de y 
\end{equation}
for $x,z\in G$, $f\in L^1(G,\RUCb( G,\Dc_0))$, and 
$\xi\in L^2( G,L^2( G,\Hc_0))\simeq L^2( G\times G,\Hc_0)$. 
\end{remark}

\begin{proposition}\label{R}
The following assertions hold: 
\begin{enumerate}
\item\label{R_item1} 
There exists an isometric $*$-homomorphism 
$$R\colon L^1(G,\RUCb(G,\Dc_0);\alpha)\to\Kern(G,\Dc_0),
\quad (Rf)(x,y)=f(xy^{-1},x). $$
\item\label{R_item2} 
For every $K\in\Ran R$ and $x,y\in G$ we have 
$(R^{-1}K)(x,y)=K(y,x^{-1}y)$. 
\item\label{R_item4} 
If $ G$ is a discrete group, then $\Ran R=\Kern(G,\Dc_0)$. 
\item\label{R_item3} 
We have $\Ran R\supseteq \Kern(G,\Dc_0)\cap\Cc(G\times G,\Dc_0)$ 
if and only if the group $G$ is either compact or discrete. 
\end{enumerate}
\end{proposition}

\begin{proof}
First note that for $f\in L^1(G,\RUCb( G,\Dc_0);\alpha)$ and $x,y\in G$ we have 
$$\Vert Rf(x,y)\Vert_{\Dc_0}
=\Vert f(xy^{-1},x)\Vert_{\Dc_0}
\le\Vert f(xy^{-1})\Vert_{\RUCb( G,\Dc_0)}.$$
Since $\Vert f(\cdot)\Vert_{\RUCb( G,\Dc_0)}\in L^1( G)$, we see that 
$Rf\in\Kern( G,\Dc_0)$ and 
$$\Vert Rf\Vert_{\Kern( G,\Dc_0)}
\le\int\limits_ G\Vert f(z)\Vert_{\RUCb( G,\Dc_0)}\de z
=\Vert f\Vert_{L^1(G,\RUCb( G,\Dc_0);\alpha)}. $$
Moreover, since the group $ G$ is unimodular, we obtain by equation~\eqref{cov_eq2} 
that for all $a,b\in G$ we have $f^*(a,b)=f(a^{-1},a^{-1}b)^*$, 
and therefore for all $x,y\in G$ we have 
$$\begin{aligned}
R(f^*)(x,y)
&=f^*(xy^{-1},x)
=f((xy^{-1})^{-1},(xy^{-1})^{-1}x)^*
=f(yx^{-1},y)^* \\
&=(Rf)^*(x,y)
\end{aligned}$$
using the involution in $\Kern(G,\Dc_0)$. 
When applied for the $C^*$-dynamical system $(\RUCb( G,\Dc_0), G,\alpha)$, 
formula \eqref{cov_eq1} gives the following expression for the product in 
$L^1(G,\RUCb( G,\Dc_0);\alpha)$: 
$$(f\star h)(x,z)=\int\limits_ G f(y,z)h(y^{-1}x,y^{-1}z)\de y.$$
Therefore we obtain 
\allowdisplaybreaks
\begin{align}
(R(f\star h))(x,z)
&=(f\star h)(xz^{-1},x) \nonumber\\
&=\int\limits_ G f(y,x)h(y^{-1}xz^{-1},y^{-1}x)\de y \nonumber\\
&=\int\limits_ G f(xv^{-1},x)h(vz^{-1},v)\de v \nonumber\\
&=\int\limits_ G (Rf)(x,v)(Rh)(v,z)\de v \nonumber\\
&=(Rf\star Rh)(x,z). \nonumber
\end{align}
So far we have proved that 
$R\colon L^1(G,\RUCb(G,\Dc_0);\alpha)\to\Kern( G,\Dc_0)$ is a contractive $*$-homomorphism. 
It is clear from the definition that $\Ker R=\{0\}$. 

It is easily checked that $(R^{-1}K)(x,y)=K(y,x^{-1}y)$ for any $K\in\Kern( G,\Dc_0)$. 
For any function $\beta\in L^1( G)$ with 
$\Vert K(v,w)\Vert\le\vert\beta(vw^{-1})\vert$ for $v,w\in G$, 
then $\Vert(R^{-1}K)(x,y)\Vert\le\vert\beta(x)\vert$ for $x\in G$, 
hence $\Vert R^{-1}K\Vert_{L^1( G,\RUCb( G,\Dc_0)}\le\Vert\beta\Vert_{L^1( G)}$. 
Therefore $\Vert R^{-1}K\Vert_{L^1( G,\RUCb( G,\Dc_0)}\le\Vert K\Vert_{\Kern( G,\Dc_0)}$, 
and in view of what we have already proved, it follows that $R$ is an isometry. 
This completes the proof of Assertions \eqref{R_item1}--\eqref{R_item2}. 

For Assertion~\eqref{R_item4} we first note that if the group $G$ is discrete, 
then $\RUCb(G,\Dc_0)=\ell^\infty(G,\Dc_0)$. 
In fact, since $G$ is discrete, for every topological space $Y$ 
and every function $\phi\colon G\to Y$ it follows that $\phi$ is continuous. 

Now, to prove that we have not only $\Ran R\subseteq\Kern(G,\Dc_0)$, 
but rather the equality there, 
let $K\in\Kern(G,\Dc_0)$ arbitrary. 
Then there exists $\beta\in\ell^1(G)$ with $\Vert K(x,y)\Vert\le\vert\beta(xy^{-1})\vert$ for all $x,y\in G$. 
Then the function $f\colon G\to\ell^\infty(G,\Dc_0)=\RUCb(G,\Dc_0)$, $(f(x))(y)=K(y,x^{-1}y)$, 
is well defined and has the property 
$$(\forall x\in G)\quad \Vert f(x)\Vert_{\ell^\infty(G,\Dc_0)}\le \vert\beta(x)\vert$$ 
hence $f\in\ell^1(G,\ell^\infty(G,\Dc_0))$. 
This implies $f\in L^1(G,\RUCb(G,\Dc_0);\alpha)$, and 
we have $Rf=K$ by the formulas that define $R$ and $f$.  
Hence $K\in\Ran R$, and this completes the proof of the equality $\Ran R\subseteq\Kern(G,\Dc_0)$. 

For Assertion~\eqref{R_item3}, it is well-known that if $G$ is either discrete or compact, 
then we have $\RUCb(G,\Dc_0)=\Cc(G,\Dc_0)\cap \ell^\infty(G,\Dc_0)$. 
(See the proof of Assertion~\eqref{R_item4} above for the discrete case.) 
If $G$ is compact, then we obtain a natural linear inclusion map 
$$\Cc(G\times G,\Dc_0)=\Cc(G\times G,\Dc_0)\cap\ell^\infty(G\times G,\Dc_0)\hookrightarrow L^1(G,\RUCb(G,\Dc_0);\alpha)$$ 
and it is easily seen that $R(\Cc(G\times G,\Dc_0))=\Cc(G\times G,\Dc_0)$ 
hence, by applying $R$ to the above inclusion, we obtain 
$\Cc(G\times G,\Dc_0)\subseteq\Ran R$. 
Since $\Ran R\subseteq\Kern(G,\Dc_0)$, it follows that $\Cc(G\times G,\Dc_0)\cap\Kern(G,\Dc_0)\subseteq\Ran R$. 
On the other hand, if $G$ is discrete, then the assertion follows by Assertion~\eqref{R_item4}

Conversely, if we have $\Ran R\supseteq \Kern(G,\Dc_0)\cap\Cc(G\times G,\Dc_0)$, 
then  
\begin{equation}\label{R_proof_eq1}
\RUCb(G,\Dc_0)=\Cc(G,\Dc_0)\cap \ell^\infty(G,\Dc_0).
\end{equation} 
In fact, the inclusion $R(L^1(G,\RUCb( G,\Dc_0);\alpha))\supseteq \Kern(G,\Dc_0)\cap\Cc(G\times G,\Dc_0)$ 
implies $L^1(G,\RUCb(G,\Dc_0);\alpha)\supseteq R^{-1}(\Kern(G,\Dc_0)\cap\Cc(G\times G,\Dc_0))$. 
For arbitrary $\phi\in\Cc(G,\CC)\cap L^1(G,\CC)$ and $\psi\in\Cc(G,\Dc_0)\cap\ell^\infty(G,\Dc_0)$  
let us define $K(x,y)=\phi(xy^{-1})\psi(x)$ for all $x,y\in G$. 
Then $K\in\Kern(G,\Dc_0)\cap\Cc(G\times G,\Dc_0)$, hence $R^{-1}K\in L^1(G,\RUCb(G,\Dc_0);\alpha)$, 
that is, $\phi\otimes\psi\in L^1(G,\RUCb(G,\Dc_0);\alpha)$. 
If we pick $\phi\in\Cc(G,\CC)\cap L^1(G,\CC)$ with $\int_G\phi\ne0$ 
and integrate the function $\phi\otimes\psi=\phi(\cdot)\psi\in L^1(G,\RUCb(G,\Dc_0);\alpha)$ 
then the integral of that function belongs to $\RUCb(G,\Dc_0)$, 
while on the other hand the integral is equal to 
$(\int_G\phi)\psi$, hence $\psi\in\RUCb(G)$. 
Since $\psi\in\Cc(G,\Dc_0)\cap\ell^\infty(G,\Dc_0)$ is arbitrary, we obtain \eqref{R_proof_eq1}. 

Since the $C^*$-algebra $\Dc_0$ is unital, 
it is straightforward to show that \eqref{R_proof_eq1} 
implies $\RUCb(G,\CC)=\Cc(G,\CC)\cap \ell^\infty(G,\CC)$, and then 
it follows by \cite[Th.~2.8]{CR66} along with \cite[Cor.~2]{Ki62} 
that the group $G$ is either discrete or compact. 
This completes the proof. 
\end{proof}

\subsection*{Some inverse-closed algebras of integral operators}
The isometric $*$-homo\-mor\-phi\-sm constructed in Proposition~\ref{R} can be used to establish a close relationship between the $\pi$-regular representation~$\Pi$ 
described in Remark~\ref{syst} and the representation that occurs in Proposition~\ref{mult}. 
In particular, by using the following result along with the faithful $*$-representation referred to in Proposition~\ref{mult}, 
one obtains an alternative proof for the fact that the mapping $R$ from Proposition~\ref{R}\eqref{R_item1} 
is a $*$-homomorphism, 
since the $\pi$-regular representation~\eqref{syst_eq0} is. 

\begin{proposition}\label{multiple}
If $G$ is a unimodular group, then there exists a unitary operator 
$$W\colon L^2( G\times G,\Hc_0)\simeq 
L^2( G,\Hc_0)\botimes L^2(G)\to L^2(G,L^2(G,\Hc_0))
\simeq L^2( G\times G,\Hc_0)$$ 
such that $(W\xi)(x,z)=\xi(xz,z)$ for $x,z\in G$ and 
$\xi\in L^2(G\times G,\Hc_0)$.  
Moreover, for every $f\in L^1(G,\RUCb(G,\Dc_0);\alpha)$ the diagram 
$$\xymatrix{
L^2( G,\Hc_0)\botimes L^2( G) \ar[r]^{W} \ar[d]_{T_{R(f)}\otimes\id_{L^2(G)}} 
& L^2( G,L^2( G,\Hc_0)) \ar[d]^{\Pi(f)} \\
L^2( G,\Hc_0)\botimes L^2( G) \ar[r]^{W} & L^2( G,L^2( G,\Hc_0))
}$$
is commutative.  
Here $\Pi$ is the $\pi$-regular representation of \eqref{syst_eq0}--\eqref{syst_eq1}. 
\end{proposition}

\begin{proof}
It is easily seen that the operator $W$ in the statement is unitary. 
Moreover, for $f\in L^1(G,\RUCb(G,\Dc_0);\alpha)$ and $\xi\in L^2(G\times G,\Hc_0)$ we have 
$$((T_{R(f)}\otimes\id_{L^2(G)})\xi)(v,u)=\int\limits_ G f(vw^{-1},v)\xi(w,u)\de w
=\int\limits_ G f(y,v)\xi(y^{-1}v,u)\de y,$$
hence 
$$\begin{aligned}
(W(T_{R(f)}\otimes\id_{L^2(G)})\xi)(x,z)
&=((T_{R(f)}\otimes\id)\xi)(xz,z) 
=\int\limits_ G f(y,xz)\xi(y^{-1}xz,z) \de y \\
&=\int\limits_ G f(y,xz)(W\xi)(y^{-1}x,z) \de y 
=(\Pi(f)W\xi)(x,z),
\end{aligned}$$
where the latter equality follows from~\eqref{syst_eq1}. 
\end{proof}

\begin{remark}\label{history}
\normalfont
The proofs of our Propositions \ref{R}\eqref{R_item1} and \ref{multiple} 
are partially based on some ideas related to the Takai duality theorem on crossed products of $C^*$-algebras; 
compare for instance the proof of \cite[Th. 7.7.12]{Pe79} or \cite{Ta75}. 
The special case of the discrete groups was also independently treated along the same lines in 
\cite[Prop. 1 and 3]{FGL08}. 
We also mention that crossed products involving the $C^*$-algebras of uniformly continuous bounded functions 
on a locally compact group, along with the related algebras of integral operators, 
were studied in \cite{Ge11} in connection with some problems in the spectral theory. 
See also Section~\ref{Sect5} below. 
\end{remark}

\begin{lemma}\label{biller}
Let $\Ac$ be an involutive symmetric Banach algebra 
and denote by $\rho_0\colon\Ac\to C^*(\Ac)$ the canonical homomorphism into the universal $C^*$-algebra of $\Ac$. 
Assume we have fixed a faithful $*$-representation $C^*(\Ac)\hookrightarrow\Bc(\Hc_1)$ 
satisfying the condition that if $\Ac$ has a unit element, then $\id_{\Hc_1}\in C^*(\Ac)$, 
and if $\Ac$ has no unit element, then $\id_{\Hc_1}\not\in C^*(\Ac)$. 
Also assume that $\rho\colon\Ac\to\Bc(\Hc)$ is a continuous injective $*$-representation satisfying the following condition: 
\begin{itemize}
\item There exist 
a complex Hilbert space $\Hc_2$ and a unitary operator 
$S\colon \Hc_1\to\Hc\botimes\Hc_2$
such that for every $a\in\Ac$ the diagram 
$$\xymatrix{
\Hc_1 \ar[r]^{S} \ar[d]_{\rho_0(a)} 
& \Hc\botimes\Hc_2 \ar[d]^{\rho(a)\otimes\id_{\Hc_2}} \\
\Hc_1 \ar[r]^{S} & \Hc\botimes\Hc_2
}$$
is commutative. 
\end{itemize} 
Then $\CC\id_{\Hc}+\rho(\Ac)$ is an inverse-closed subalgebra of $\Bc(\Hc)$. 
\end{lemma}

\begin{proof} 
We must prove that $(\CC\id_{\Hc}+\rho(\Ac))^\times=(\CC\id_{\Hc}+\rho(\Ac))\cap\Bc(\Hc)^\times$.  
We will consider separately the two cases that can occur. 

Case $1^\circ$ Assume that there exists a unit element $\1\in\Ac$, 
hence $\id_{\Hc_1}\in C^*(\Ac)$ by the hypothesis. 
Since $\rho_0(\1)$ is the unit element of $C^*(\Ac)$, it follows by the uniqueness of the unit element 
that $\rho_0(\1)=\id_{\Hc_1}$, and then the commutative diagram from the statement shows that also $\rho(\1)=\id_{\Hc}$. 

The conclusion reduces to 
$\rho(\Ac)^\times=\rho(\Ac)\cap\Bc(\Hc)^\times$. 
Since $\rho$ is a faithful representation, 
the latter equality is equivalent to the fact that for 
$a\in\Ac$ we have 
$a\in\Ac^\times$ if and only if $\rho(a)\in \Bc(\Hc)^\times$.
In fact, we have 
$$a\in\Ac^\times\iff \rho_0(a)\in C^*(\Ac)^\times 
\iff \rho_0(a)\in \Bc(\Hc_1)^\times
\iff \rho(a)\in \Bc(\Hc)^\times,$$
where the first equivalence follows by Proposition~\ref{univ}, 
the second equivalence relies on the fact that every $C^*$-algebra of operators is inverse closed, 
and the third equivalence is a consequence of the commutative diagram in the statement. 

Case $2^\circ$ 
Now assume that $\Ac$ has no unit element, hence $\id_{\Hc_1}\not\in C^*(\Ac)$, again by the hypothesis. 
It suffices to prove that $(\CC\id_{\Hc}+\rho(\Ac))^\times\supseteq(\CC\id_{\Hc}+\rho(\Ac))\cap\Bc(\Hc)^\times$, 
since the converse inclusion always holds. 
So let $  z\in\CC$ and $a\in\Ac$ arbitrary with $  z\id_{\Hc}+\rho(a)\in\Bc(\Hc)^\times$. 
We will show that $(  z\id_{\Hc}+\rho(a))^{-1}\in\CC\id_{\Hc}+\rho(\Ac)$. 

To this end let $\Ac_1:=\CC\1\dotplus\Ac$ be Banach $*$-algebra obtained as the unitization of $\Ac$, 
with the canonical unital $*$-homomorphism into its enveloping $C^*$-algebra $\rho_1\colon\Ac_1\to C^*(\Ac_1)$. 
Also let $\iota\colon\Ac\hookrightarrow\Ac_1$ be the canonical inclusion map. 
It follows by the functorial property of enveloping $C^*$-algebras (see for instance \cite[Th. 10.1.11(c)]{Pl01}) 
that there exists a unique $*$-homomorphism $\phi$ for which the diagram 
$$\xymatrix{
\Ac \ar[r]^{\iota} \ar[d]_{\rho_0} & \Ac_1 \ar[d]^{\rho_1} \\
C^*(\Ac) \ar[r]^{\phi} & C^*(\Ac_1)
}$$
is commutative. 
By using the commutative diagram from the statement and then the fact that 
unital $C^*$-algebras of operators are inverse closed, we then obtain  
$$\begin{aligned}
  z\id_{\Hc}+\rho(a)\in\Bc(\Hc)^\times
& \Rightarrow    z\id_{\Hc_1}+\rho_0(a)\in\Bc(\Hc_1)^\times \\
& \Rightarrow   z\id_{\Hc_1}+\rho_0(a)\in (\CC\id_{\Hc_1}+C^*(\Ac))^\times.
\end{aligned}$$
Now, by using the hypothesis $\id_{\Hc_1}\not\in C^*(\Ac)^\times$, 
it follows that the $C^*$-algebra $C^*(\Ac)_1:=\CC\id_{\Hc_1}\dotplus C^*(\Ac)$ is the unitization of $C^*(\Ac)$. 
Therefore, if we 
extend $\phi$  to  $C^*(\Ac)_1$ by $\phi(\id_{\Hc_1})=\1\in C^*(\Ac_1)$, 
then we further obtain 
$$\begin{aligned} 
  z\id_{\Hc_1}+\rho_0(a)\in (C^*(\Ac)_1)^\times
& \Rightarrow
\phi(  z\id_{\Hc_1}+\rho_0(a))\in C^*(\Ac_1)^\times \\
& \Rightarrow \rho_1(  z\1+a) \in C^*(\Ac_1)^\times \\
& \Rightarrow   z\1+a \in \Ac_1^\times
\end{aligned}$$
where the first implication holds since unital $*$-homomorphisms map invertible elements into invertible elements, 
the second implication relies on the above commutative diagram, 
and the third implication follows by using Proposition~\ref{univ} 
for the symmetric unital Banach $*$-algebra $\Ac_1$ with its canonical $*$-homomorphism 
$\rho_1\colon\Ac_1\to C^*(\Ac_1)$. 

Finally, by $z\1+a \in \Ac_1^\times$ we obtain 
$(  z\id_{\Hc}+\rho(a))^{-1}\in\CC\id_{\Hc}+\rho(\Ac)$, 
and this completes the proof. 
\end{proof}

Unlike other approaches to constructing inverse-closed algebras, 
Lemma~\ref{biller} above relies neither on Hulanicki's lemma \cite[Prop. 2.5]{Hu72} 
which requires computations of spectral radii, 
nor on its later improvements along the same lines.

\begin{theorem}\label{main}
Let the group $ G$ be unimodular, amenable, such that its discrete undelying group $G_d$ is rigidly symmetric, and 
denote  
$\Ac_{ G,\Dc_0}:=\{T_K\mid K\in\Ran R\}$, 
using the notation of Proposition~\ref{R}. 
Then $\CC\1+\Ac_{G,\Dc_0}$ is an inverse-closed involutive subalgebra of $\Bc(L^2(G,\Hc_0))$. 
If moreover $ G$ is a discrete group, then  
$\1\in\Ac_{G,\Dc_0}=\{T_K\mid K\in\Kern(G,\Dc_0)\}$. 
\end{theorem}

\begin{proof}
Let us define 
$$\rho\colon L^1(G,\RUCb(G,\Dc_0);\alpha)\to\Bc(L^2( G,\Hc_0)),\quad 
f\mapsto T_{R(f)}. $$
Since the group $ G$ is unimodular, 
it follows by Propositions \ref{R}~and~\ref{mult} 
that the mapping $\rho$ is a continuous injective $*$-homomorphism. 

Moreover, since $ G$ is a rigidly symmetric group, Corollary~\ref{leptin2} 
shows that the covariance algebra $L^1(G,\RUCb(G,\Dc_0);\alpha)$ is 
a symmetric involutive Banach algebra. 
On the other hand, the universal enveloping $C^*$-algebra 
$C^*(L^1(G,\RUCb(G,\Dc_0);\alpha))$ is isomorphic to the norm-closure of 
the range of the regular representation 
$$\Pi\colon L^1(G,\RUCb(G,\Dc_0);\alpha)\to\Bc(L^2(G\times G,\Hc_0)),$$  
since $ G$ is an amenable group; see for instance~\cite[Th. 7.13]{Wi07}. 
The corresponding canonical homomorphism 
$$\rho_0\colon L^1(G,\RUCb(G,\Dc_0);\alpha)\to C^*(L^1(G,\RUCb(G,\Dc_0);\alpha)),\quad f\mapsto\Pi(f)$$
is the one induced by $\Pi$. 

We have thus checked that the hypothesis of Lemma~\ref{biller} is satisfied,   
by using Proposition~\ref{multiple} 
and the following facts: 
If the group $G$ is discrete, then the Banach $*$-algebra $L^1(G,\RUCb(G,\Dc_0);\alpha)$ is unital 
and $\rho_0$ is a unital $*$-homomorphism, 
while if $G$ fails to be discrete, then $L^1(G,\RUCb(G,\Dc_0);\alpha)$ has no unit element 
and also the crossed product $C^*(L^1(G,\RUCb(G,\Dc_0);\alpha))$ does not contain the identity operator 
on $L^2(G\times G,\Hc_0)$. 
Then the first part of the conclusion follows since we have 
$\Ac_{G,\Dc_0}=\rho(L^1(G,\RUCb(G,\Dc_0);\alpha))$. 
For the second assertion we use Proposition~\ref{R}\eqref{R_item4}. 
\end{proof}

\begin{example}\label{pol}
\normalfont
The hypothesis of Theorem~\ref{main} is satisfied if $G$ is a locally compact nilpotent group. 
Indeed, as already noted after Definition~\ref{sym},  
the underlying discrete  group  of such a group is rigidly symmetric by~\cite[Cor.~6]{Po92}. 
Moreover, every nilpotent group is amenable; see for instance \cite[Prop. (0.15)--(0.16)]{Pa88}. 

Recall also that locally compact nilpotent group have polynomial growth by \cite[Cor. (6.18)]{Pa88}. 
Theorem~\ref{main} applies for any discrete finitely generated group $G$ with  polynomial growth. 
Indeed, every group of that type has a nilpotent subgroup of finite index by \cite{Gr81}, 
and therefore the required rigid symmetry follows by \cite[Cor. 3]{LP79} (see also \cite[Th. 3]{Le74}). 
Moreover, the groups with polynomial growth (the locally compact ones, not necessarily discrete) 
are unimodular by \cite[Prop. (6.6), (6.9)]{Pa88} 
and amenable by \cite[Prop. (0.13)]{Pa88}.  
\end{example}

\begin{corollary}\label{baskakov}
Let $\Hc$ be a complex Hilbert space and assume that we have an orthogonal direct sum decomposition 
$\Hc=\bigoplus\limits_{\gamma\in\Gamma}\Hc_\gamma$ 
whose summands are isomorphic to each other, 
where $\Gamma$ is a 
finitely generated group with polynomial growth. 
Denote by $P_\gamma\in\Bc(\Hc)$ the orthogonal projection on $\Hc_\gamma$ for $\gamma\in\Gamma$, 
and define 
$$\Bc:=\{S\in\Bc(\Hc)\mid(\exists\beta\in\ell^1(\Gamma))\quad 
\Vert P_{\gamma_1} S P_{\gamma_2}\Vert\le\beta_{\gamma_1\gamma_2^{-1}} 
\text{ for }\gamma_1,\gamma_2\in\Gamma\}. $$
Then $\Bc$ is a unital subalgebra of $\Bc(\Hc)$ and we have 
$\Bc^\times=\Bc\cap\Bc(\Hc)^\times$. 
\end{corollary}

\begin{proof}
Since the Hilbert spaces in the family $\{\Hc_\gamma\}_{\gamma\in\Gamma}$ are isomorphic to each other, 
we may actually assume that there exists some complex Hilbert space $\Hc_0$ such that $\Hc_\gamma=\Hc_0$ for every $\gamma\in\Gamma$. 
Then $\Hc=\ell^2(\Gamma)\botimes\Hc_0=\ell^2(\Gamma,\Hc_0)$. 
It follows by Example~\ref{pol} that the group $\Gamma$ is amenable and rigidly symmetric, 
hence Theorem~\ref{main} applies with $\Dc_0=\Bc(\Hc_0)$.  
Specifically, the formula 
$$(\forall S\in\Bc(\Hc))(\forall \gamma_1,\gamma_2\in\Gamma)\quad 
K_S(\gamma_1,\gamma_2)=P_{\gamma_1}S\vert_{\Ran P_{\gamma_2}}$$  
defines an algebra isomorphism $\Bc\mathop{\to}\limits^\sim\Kern(G,\Dc_0)$, $S\mapsto K_S$,  
which is actually the inverse of the canonical representation
$\Kern(G,\Dc_0)\mathop{\to}\limits^\sim\Ac_{G,\Dc_0}$, $K\mapsto T_K$ 
(see~\eqref{intop} and Proposition~\ref{mult}). 
We thus obtain $\Bc=\Ac_{G,\Dc_0}$,  and the conclusion follows by Theorem~\ref{main}. 
\end{proof}

\begin{remark}
\normalfont
The special case Corollary~\ref{baskakov} when $\Gamma$ is an abelian group 
is also a special case of \cite[Cor.~1 to Th.~2]{Ba97}. 
On the other hand, the special case of Corollary~\ref{baskakov} when $\dim\Hc_\gamma=1$ 
for every $\gamma\in\Gamma$ was obtained in~\cite{FGL08}. 
\end{remark}

\section{Larger inverse-closed algebras of integral operators}\label{Sect3}

For the sake of completeness, we will sketch here a procedure that allows one to construct, 
on some special topological groups, 
inverse-closed algebras of integral operators 
that are larger than the algebra from Theorem~\ref{main}. 


\begin{lemma}\label{L1}
If $\Ac$ is a unital associative algebra and $\Jc$ is a left ideal in $\Ac$ 
(that is, $\Ac\Jc\subseteq\Jc$), then 
$\CC\1+\Jc$ is an inverse-closed subalgebra of $\Ac$. 
\end{lemma}

\begin{proof}
To prove that $(\CC\1+\Jc)^\times\supseteq(\CC\1+\Jc)\cap\Ac^\times$, 
let $a\in (\CC\1+\Jc)\cap\Ac^\times$ arbitrary. 
Then there exist $\alpha\in\CC$, $a_0\in\Jc$ and $b\in\Ac$ such that 
$a=\alpha\1+a_0$ and $ab=ba=\1$. 
Then $b(\alpha\1+a_0)=\1$, 
hence $\alpha b=\1-ba_0\in\CC\1+\Ac\Jc\subseteq\CC\1+\Jc$. 

Now, if $\Jc=\Ac$, then the assertion is trivial. 
Otherwise, $\Jc$ cannot contain invertible elements, hence $\alpha\ne0$, and then the above relation implies 
$b(=a^{-1})\in\CC\1+\Jc$, hence $a\in(\CC\1+\Jc)^\times$. 
\end{proof}

In the Step~$2^\circ$ of the proof of the following lemma we use an idea from the proof of \cite[Th.~2.13(3)]{BB12}. 
See Remark~\ref{referee2} for an alternative proof and additional information on specific applications. 

\begin{lemma}\label{L2}
Let $\Bc$ be any unital associative Banach algebra. 
Assume that $\Ac_0$ is another associative Banach algebra such that there exists a continuous injective homomorphism $\Ac_0\hookrightarrow\Bc$. 
Moreover, let $\Jc$ be a left ideal of $\Ac_0$ with the following properties: 
\begin{itemize}
\item $\Jc$ is dense in $\Ac_0$; 
\item $(\CC\1+\Jc)^\times=(\CC\1+\Jc)\cap\Bc^\times$. 
\end{itemize}
Then we have $(\CC\1+\Ac_0)^\times
=(\CC\1+\Ac_0)\cap\Bc^\times$.
\end{lemma}

\begin{proof}
The inclusion ``$\subseteq$'' is clear. 

To prove the converse inclusion ``$\supseteq$'', let $\alpha\1+a_0\in(\CC\1+\Ac_0)\cap\Bc^\times$ arbitrary, 
where $\alpha\in\CC$ and $a_0\in\Ac_0$. 
We will proceed in two steps. 

Step $1^\circ$ 
First assume $\alpha\ne0$.  
Then we may (and do) assume $\alpha=1$. 

The ideal $\Jc$ is dense in $\Ac_0$, hence there exists $\bar a\in\Jc$ with $\Vert\bar a-a_0\Vert_{\Ac_0}<1/2$. 
Then $\1+(a_0-\bar a)\in(\CC\1+\Ac_0)^\times$, so we can consider the element 
$$k:=(\1+(a_0-\bar a))^{-1}(\1+a_0)
=(\1+(a_0-\bar a))^{-1}(\1+(a_0-\bar a)+\bar a)
=\1+(\1+(a_0-\bar a))^{-1}\bar a. $$
Here $k\in(\CC\1+\Ac_0)^\times$ and $\bar a\in\Jc$, 
hence actually 
$$k\in(\CC\1+\Ac_0)^\times\cap(\CC\1+\Jc)=(\CC\1+\Jc)^\times,$$ 
where the latter equality follows by Lemma~\ref{L1} for $\Ac=\CC\1+\Ac_0$. 

Thus $k^{-1}\in\CC\1+\Jc$, and then we get by the definition of $k$ that 
$$(\1+a_0)^{-1}=k^{-1}(\1+(a_0-\bar a))^{-1}\in\CC\1+\Ac_0$$ 
and this concludes the first step of the proof. 

Step $2^\circ$ 
Now assume $\alpha=0$, hence $a_0\in(\CC\1+\Ac_0)\cap\Bc^\times$. 
Since $\Bc$ is a Banach algebra, it follows that its set of invertible elements $\Bc^\times$ is 
an open subset, hence there exists $\varepsilon>0$ for which 
for every $\alpha\in\CC$ with $\vert\alpha\vert\le\varepsilon$ we have 
$\alpha\1+a_0\in(\CC\1+\Ac_0)\cap\Bc^\times$. 
Then by the conclusion of Step~$1^\circ$ above we have 
$(\alpha\1+a_0)^{-1}\in \CC\1+\Ac_0$ if $\alpha\in\CC$ with $\vert\alpha\vert=\varepsilon$. 
Now we may use the holomorphic functional calculus in the unital Banach algebra $\Bc$ 
to write 
$$a_0^{-1}=\frac{1}{2\pi\ie}\int\limits_{\vert\alpha\vert=\varepsilon}(\alpha\1+a_0)^{-1}\de\alpha$$
and this implies $a_0^{-1}\in\CC\1+\Ac_0$
since the integral from the right-hand side is convergent both in the norm of $\Bc$ and in the norm of $\CC\1+\Ac_0$, 
and the values of its integrand belong to $\CC\1+\Ac_0$. 

Consequently $(\CC\1+\Ac_0)^\times\supseteq(\CC\1+\Ac_0)\cap\Bc^\times$, and we are done. 
\end{proof}

For the following theorem we recall the notation introduced in Definition~\ref{k1}. 
We also denote by $\Ac(G)$ the image of the canonical representation~\eqref{intop}
for $\Dc_0=\CC$. 
It follows by Proposition~\ref{mult} that $\Ac(G)$ has a natural structure of involutive associative Banach algebra.  
Specifically, $\Ac(G)$ is endowed with the norm obtained by transporting the norm of 
the Banach algebra of integral kernels $\Kern(G,\CC)$ via the canonical representation~\eqref{intop}, 
which is thus turned into an isometric $*$-isomorphism $\Kern(G,\CC)\mathop{\to}\limits^\sim\Ac(G)$. 
Since the canonical $*$-representation is contractive by Proposition~\ref{mult}, 
we see that the inclusion map $\Ac(G)\hookrightarrow\Bc(\Hc)$ 
is a contractive $*$-homomorphism of involutive Banach algebras.

\begin{theorem}\label{kurbatov}
Let $G$ be any locally compact group. 
For every linear subspace $\Fc\subseteq\Kern(G,\CC)$ define 
$$\Ac_{\Fc}(G):=\{T_K\mid K\in\Kern_{\Fc}(G,\CC)\}\subseteq\Bc(L^2(G)).$$
Then the following assertions hold true: 
\begin{enumerate}
\item\label{kurbatov_item1} If $\Fc$ is any subalgebra (right/left/two-sided ideal, respectively) 
of the convolution algebra $L^1(G)$, 
then $\Ac_{\Fc}(G)$ is a subalgebra (right/left/two-sided ideal, respectively) 
of the algebra of integral operators $\Ac(G)\subseteq \Bc(L^2(G))$. 
\item\label{kurbatov_item2} If $\CC\1+\Ac(G)$ is an inverse-closed subalgebra of $\Bc(L^2(G))$, 
then $\Ac_{\Fc}(G)$ is an inverse-closed subalgebra of $\Bc(L^2(G))$ for every right/left 
ideal $\Fc$ of $L^1(G)$.  
\item\label{kurbatov_item3} Conversely, if there exists a right/left ideal $\Fc$ of $L^1(G)$ 
for which the subalgebra $\CC\1+\Ac_{\Fc}(G)$ of $\Bc(L^2(G))$ is inverse closed 
and for all $0\le\beta\in L^1(G)$ there exists a sequence $\{\beta_n\}_{n\ge 1}$ in $\Fc$ 
with $0\le\beta_n\le\beta$ and $\lim\limits_{n\to\infty}\beta_n=\beta$ almost everywhere on~$G$, 
then $\CC\1+\Ac(G)$ is inverse closed in $\Bc(L^2(G))$. 
\end{enumerate}
\end{theorem}

\begin{proof}
 For Assertion~\eqref{kurbatov_item1}, just note that if $K_j\in\Kern(G,\CC)$, $0\le\beta_j\in L^1(G)$, 
and $\vert K_j(x,y)\vert\le\beta_j(xy^{-1})$ for almost all $x,y\in G$ and $j=1,2$, 
then 
$$\vert (K_1\star K_2)(x,y)\vert\le(\beta_1\star\beta_2)(xy^{-1}) 
\text{ and }
\vert (K_1+ K_2)(x,y)\vert\le(\beta_1+\beta_2)(xy^{-1})$$ 
for almost all $x,y\in G$. 

Assertion~\eqref{kurbatov_item2} follows by Lemma~\ref{L1} and its obvious version for right ideals. 

For Assertion~\eqref{kurbatov_item3} we may assume that $\Fc$ is a left ideal, 
since the case of right ideals can be treated similarly. 
Let $K\in\Kern(G,\CC)$ and $0\le\beta\in L^1(G)$ arbitrary with 
$\vert K(x,y)\vert\le\vert\beta(xy^{-1})\vert$ 
for almost all $x,y\in G$. 
By the hypothesis, there exists a sequence $\{\beta_n\}_{n\ge 1}$ in $\Fc$ 
with $0\le\beta_n\le\beta$ and $\lim\limits_{n\to\infty}\beta_n=\beta$ almost everywhere on~$G$. 
For every $n\ge1$, let $a_n$ be any measurable function on $G$ with $0\le a_n\le 1$ 
and $\beta_n=a_n\beta$. 
The function $a_n$ is uniquely determined almost everywhere on the set where $\beta$ does not vanish. 
Now define 
$$K_n\colon G\times G\to\CC,\quad K_n(x,y)=a_n(xy^{-1})K(x,y).$$
Then we have $\vert K_n(x,y)\vert\le(a_n\beta)(xy^{-1})=\beta_n(xy^{-1})$ for almost all $x,y\in G$. 
Since $\beta_n\in\Fc$, it follows that $K\in\Kern_{\Fc}(G,\CC)$. 

On the other hand, for almost all $x,y\in G$ we have 
$$\vert(K-K_n)(x,y)\vert=\vert(1-a_n)(xy^{-1})K(x,y)\vert\le((1-a_n))\beta)(xy^{-1})=(\beta-\beta_n)(xy^{-1}) $$
and this implies $\Vert K-K_n\Vert_{\Kern(G,\CC)}\le\Vert\beta-\beta_n\Vert_{L^1(G)}$. 
The conditions satisfied by $\{\beta_n\}_{n\ge 1}$ entail that 
$\lim\limits_{n\to\infty}\Vert\beta-\beta_n\Vert_{L^1(G)}=0$ 
by Lebesgue's dominated convergence theorem, 
and therefore $\lim\limits_{n\to\infty}\Vert K-K_n\Vert_{\Kern(G,\CC)}=0$. 

Consequently, $\Kern_{\Fc}(G,\CC)$ is dense in $\Kern(G,\CC)$, 
and this implies that $\Ac_{\Fc}(G)$ is dense in the Banach algebra $\Ac(G)$. 
Moreover, since $\Fc$ is a left ideal of the convolution algebra $L^1(G)$, 
it follows by Assertion~\eqref{kurbatov_item1} that $\Ac_{\Fc}(G)$ is a left ideal of $\Ac(G)$. 
Since the hypothesis ensures that $\Ac_{\Fc}(G)$ is an inverse-closed subalgebra of $\Bc(L^2(G))$, 
it follows that Lemma~\ref{L2} can be applied for 
$\Ac_0=\Ac(G)$, $\Jc=\Ac_{\Fc}(G)$, and $\Bc=\Bc(L^2(G))$, 
to obtain that $\Ac(G)$ is an inverse-closed subalgebra of $\Bc(L^2(G))$. 
This completes the proof. 
\end{proof}

\begin{corollary}\label{kurbatov_cor}
Let $G$ be any locally compact group. 
If there exists a left/right ideal $\Fc\subseteq L^1(G)$ that contains every function in $L^\infty(G)$ with compact essential support, 
and for which 
$\CC\1+\Ac_{\Fc}(G)$ is an inverse-closed subalgebra of $\Bc(L^2(G))$, 
then also $\CC\1+\Ac(G)$ is an inverse-closed subalgebra of $\Bc(L^2(G))$. 
\end{corollary}

\begin{proof}
Let $K\in\Kern(G,\CC)$ and $0\le \beta\in L^1(G)$ arbitrary with 
$\vert K(x,y)\vert\le \beta(xy^{-1})$ 
for almost all $x,y\in G$. 
Since the finite Borel measure $\beta(x)\de x$ is regular on $G$, 
it follows that for every $\varepsilon>0$ 
there exists a compact set $E_\varepsilon\subseteq G$
for which 
\begin{equation}\label{kurbatov_cor_proof_eq1}
0\le \int\limits_{G\setminus E_\varepsilon}\beta(x)\de x<\varepsilon.
\end{equation}
Now pick any continuous function with compact support $\phi_\varepsilon\colon G\to[0,1]$ 
with $\phi_\varepsilon\vert_{E_\varepsilon}\equiv1$, and define $\beta_\varepsilon=\phi_\varepsilon\beta$, 
so that $0\le\beta_\varepsilon\le\beta$, the support of $\beta_\varepsilon$ is compact, 
and by \eqref{kurbatov_cor_proof_eq1} also $\Vert\beta-\beta_\varepsilon\Vert_{L^1(G)}<\varepsilon$. 

Now for every $n\ge 1$ define $\beta_{\varepsilon,n}=\min\{\beta_\varepsilon,n\}$, 
so that $0\le\beta_{\varepsilon,n}\le\beta_\varepsilon\le\beta$ 
and $\lim\limits_{n\to\infty}\beta_{\varepsilon,n}=\beta_\varepsilon$ almost everywhere on~$G$. 
If $L^p_{\rm comp}(G)$ denotes the set of all functions in $L^p(G)$ with compact essential support, 
then $\beta_\varepsilon\in L^1_{\rm comp}(G)$, 
hence $\beta_{\varepsilon,n}\in L^\infty_{\rm comp}(G)\subseteq\Fc$ for all $n\ge 1$. 
Moreover, by Lebesgue's dominated convergence theorem we have 
$\lim\limits_{n\to\infty}\Vert\beta_{\varepsilon,n}-\beta_\varepsilon\Vert_{L^1(G)}=0$, 
hence there exists $n_\varepsilon\ge1$ with 
$\Vert\beta_{\varepsilon,n_\varepsilon}-\beta_\varepsilon\Vert_{L^1(G)}<\varepsilon$. 

Thus for every $\varepsilon>0$ we obtained the function 
$\psi_\epsilon:=\beta_{\varepsilon,n_\varepsilon}\in L^\infty_{\rm comp}(G)\subseteq\Fc$ 
with $0\le \psi_\varepsilon\le\beta$ and $\Vert\psi_\varepsilon-\beta\Vert_{L^1(G)}<2\varepsilon$. 
Now the sequence $\{\psi_{1/j}\}_{j\ge 1}$ is convergent to $\beta$ in $L^1(G)$, 
and then it has a subsequence which is convergent to $\beta$ almost everywhere. 
This shows that $\Fc$ satisfies the hypothesis of Theorem~\ref{kurbatov}\eqref{kurbatov_item3}, 
and an application of that theorem completes the present proof. 
\end{proof}

\begin{example}\label{amalgam}
\normalfont
Let $G$ be an abelian locally compact group. 
It was proved in \cite{Ku99}  and \cite{Ku01} that 
the hypothesis of the above Corollary~\ref{kurbatov_cor} 
is satisfied if $\Fc$ is the Wiener amalgam space $W(L^\infty,\ell^1)$ of \cite{Fe83}. 
The same property was then established for the reduced Heisenberg groups in \cite{FS10}. 
\end{example}

\begin{remark}\label{referee2}
\normalfont 
We will indicate here an alternative proof for Step~$2^\circ$ of Lema~\ref{L2} 
which is relevant for explaining the applicability of Theorem~\ref{kurbatov}, 
as discussed below. 

By using the notation from the aforementioned proof, 
if $\alpha=0$ then $a_0\in\Ac_0\cap\Bc^\times$. 
Since the inclusion map $\Ac_0\hookrightarrow \Bc$ is continuous 
and $\Bc^\times$ is open in $\Bc$, 
it follows that $\Ac_0\cap\Bc^\times$ is an open subset of $\Ac_0$.  
On the other hand, $\Jc$ is dense in $\Ac_0$ hence there exists $a_I\in\Jc\cap(\Ac_0\cap\Bc^\times)=\Jc\cap\Bc^\times\subseteq(\CC\1+\Jc)\cap\Bc^\times$. 
Then by the second hypothesis of Lema~\ref{L2} we obtain $a_I^{-1}\in\CC\1+\Jc\subseteq\CC\1+\Ac_0$. 
Thus $\CC\1+\Jc$ is a left ideal of $\CC\1+\Ac_0$ which contains invertible elements. 
Then it is well known and easy to check that $\CC\1+\Jc=\CC\1+\Ac_0$, that is, $\Jc=\Ac_0$, 
and in this case the conclusion coincides with the second hypothesis from the statement of the lemma. 
This ends the alternative proof of Step~$2^\circ$ of Lema~\ref{L2}. 

As a by-product of the above reasoning, 
in the setting of Lema~\ref{L2}, if we have $\Jc\subsetneqq\Ac_0$ 
(that is, the conclusion does not coincide with the second hypothesis) 
then necessarily $\Ac_0\cap\Bc^\times=\emptyset$, and in particular $\1\in\Bc\setminus\Ac_0$. 
In the special case of Theorem~\ref{kurbatov}\eqref{kurbatov_item3}, 
where $\Ac_0=\Ac(G)$, if the conclusion is nontrivial (i.e., different from the hypothesis) 
then $\Ac(G)$ cannot contain any invertible operator on $L^2(G)$, 
since otherwise every $\Fc$ which satisfies the hypothesis actually has the property $\Ac_{\Fc}(G)=\Ac(G)$. 
In particular, the latter situation happens if $G$ is a discrete group since the identity operator on $L^2(G)$ 
does belong to $\Ac(G)$ in that case. 

The above observations illustrate the fact that the interest in Theorem~\ref{kurbatov}\eqref{kurbatov_item3} 
does not come from its applications to discrete groups but rather 
from the fact that in some concrete situations (see Example~\ref{amalgam}) 
it essentially reduces the study of some algebras of integral operators on Lie groups as $\RR^n$ or the reduced Heisenberg groups 
to the study of operator algebras on some of their discrete co-compact subgroups.  
The situation of discrete groups can be dealt with by using Theorem~\ref{main}. 
\end{remark}

\section{Dense inverse-closed subalgebras of some $C^*$-algebras}\label{Sect5}

In this section we apply the above results to the elliptic algebra introduced in \cite[Sect. 6]{Ge11}
for any unimodular noncompact locally compact group $G$.  
We recall that the elliptic algebra is the $C^*$-algebra $\Ec(G)\subseteq \Bc(L^2(G))$ 
generated by the operators defined by integral kernels 
$K\in\RUCb(G\times G,\CC)$ satisfying the following controllability condition: 
\begin{itemize}
 \item There exists a compact set $S_K\subseteq G$ for which $K(x,y)=0$ if $xy^{-1}\not\in S_K$. 
\end{itemize}
To explain the terminology and point out the physical significance of the elliptic algebra, 
we recall from \cite{DG04} the following fact which holds true for the abelian Lie group $G=(\RR^n,+)$: 
Let $m\ge 1$ be any integer and $h\colon\RR^n\to\RR$ be any elliptic polynomial of order~$m$. 
Then the elliptic algebra $\Ec(G)$ coincides with the $C^*$-subalgebra of $\Bc(L^2(G))$ 
generated by the resolvents of all self-adjoint operators $h(\ie\nabla)+W$, 
where $W$ runs over the set of all symmetric differential operators of order strictly less than~$m$, 
whose coefficients are smooth functions which are bounded together their derivatives of arbitrarily high order.

Now return to the general case where $G$ is any unimodular noncompact locally compact group 
and denote by $\Kc(L^2(G))$ the $C^*$-algebra of compact operators on $L^2(G)$. 
It was already noted in \cite[subsect. 6.1]{Ge11} that $\Kc(L^2(G))\subseteq\Ec(G)\simeq \RUCb(G,\CC)\rtimes_r G$. 
In Theorem~\ref{final} below we 
show that the involutive Banach algebra $\Ac_{G,\CC}$ (see Theorem~\ref{main}) 
can be used in order to complete the foregoing relationship 
to a commutative diagram 
$$\xymatrix
{\Kc(L^2(G)) \ar@{^{(}->}[r] & \Ec(G)  & \RUCb(G,\CC)\rtimes_r G \ar@{<->}[l]\\
\Ac_{G,\CC}\cap\Kc(L^2(G)) \ar@{^{(}->}[u] \ar@{^{(}->}[r]  & \Ac_{G,\CC}  \ar@{^{(}->}[u] 
& L^1(G,\RUCb(G,\CC);\alpha) \ar@{<->}[l] \ar@{^{(}->}[u] 
}$$
consisting of continuous inclusion maps and isometric $*$-isomorphisms. 
The vertical arrows in that diagram have dense ranges and are inverse-closed after the algebras 
of the above diagram have been unitized, if the group $G$ is also amenable and symmetric. 

Since it is not the case that $\1\in\Ec(G)$ if the group $G$ is nondiscrete and amenable, 
we will need to introduce the unital $C^*$-algebra 
$$\Ec_1(G):=\CC\1+\Ec(G)\subseteq \Bc(L^2(G)).$$ 
We also need the unitary representations 
$\lambda,\rho\colon G\to \Bc(L^2(G))$ 
defined by 
$$(\lambda(a)\phi)(x)=\phi(a^{-1}x)\text{ and }(\rho(a)\phi)(x)=\phi(xa)$$ 
for $\phi\in L^2(G)$ and $a,x\in G$.  
For every unitary operator $V\colon L^2(G)\to L^2(G)$ define 
$$\Ad\,V\colon\Bc(L^2(G))\to\Bc(L^2(G)),\quad (\Ad\,V)T=VTV^{-1}.$$

For Theorem~\ref{final} we need the following remark.

\begin{remark}\label{prior}
\normalfont
If the canonical representation~\eqref{intop}, which is faithful by Proposition~\ref{mult}, 
namely 
$$\Kern(G,\CC)\to\Bc(L^2(G)), \quad K\mapsto T_K$$ 
is composed with the isometric $*$-homomorphism given by Proposition~\ref{R}, 
namely 
$$R\colon L^1(G,\RUCb(G,\CC);\alpha)\to \Kern(G,\CC)$$
then we obtain a $*$-isomorphism onto the $*$-algebra $\Ac_{G,\CC}$ from Theorem~\ref{main}
$$\Psi\colon L^1(G,\RUCb(G,\CC);\alpha)\mathop{\to}\limits^\sim\Ac_{G,\CC}, \quad 
f\mapsto T_{R(f)}$$ 
which will be used for defining a norm on $\Ac_{G,\CC}$, 
thus turning it into an involutive Banach algebra 
for which $\Psi$ is an isometry.  
Since the aforementioned canonical $*$-representation is contractive by Proposition~\ref{mult}, 
we also obtain the inclusion maps  
$$\Ac_{G,\CC}\hookrightarrow\Ac(G)\hookrightarrow\Bc(\Hc),$$ 
where $\Ac_{G,\CC}\hookrightarrow\Ac(G)$ is an isometric $*$-homomorphism 
and $\Ac(G)\hookrightarrow\Bc(\Hc)$
is a contractive $*$-homomorphism of involutive Banach algebras. 
\end{remark}

\begin{theorem}\label{final}
Let $G$ by any locally compact group $G$ which is noncompact, 
unimodular, amenable, and such that its discrete undelying group $G_d$ is rigidly symmetric.
Then we have 
\begin{enumerate}
 \item\label{final_item1} There exists a continuous inclusion $\Ac_{G,\CC}\hookrightarrow\Ec(G)$. 
 \item\label{final_item2} The intersection $\Ac_{G,\CC}\cap\Kc(L^2(G))$ is a dense $*$-subalgebra of
  $\Kc(L^2(G))$.  
 \item\label{final_item3} The unitization 
 $\CC\1+\Ac_{G,\CC}$ is a dense inverse-closed $*$-subalgebra of the unital $C^*$-algebra $\Ec_1(G)$.  
 \item\label{final_item4} The group $G$ acts by isometric $*$-automorphisms on the $C^*$-algebra $\Ec(G)$ 
 by both $\Ad\,\lambda(\cdot)$ and $\Ad\,\rho(\cdot)$. 
 The Banach algebra $\Ac_{G,\CC}$ is invariant under each of these actions, 
 which give rise to actions of $G$ by isometric $*$-automorphisms of $\Ac_{G,\CC}$,
  and moreover the mapping 
\begin{equation}\label{final_eq1}
G\times \Ac_{G,\CC}\to\Ac_{G,\CC}, \quad (a,S)\mapsto(\Ad\,\lambda(a))S
\end{equation}
is continuous. 
\end{enumerate}
\end{theorem}

\begin{proof}
For Assertion~\eqref{final_item1} recall from \cite[Prop. 6.5]{Ge11} that 
the reduced crossed product $\RUCb(G)\rtimes_r G$ is $*$-isomorphic to $\Ec(G)$ 
by means of the mapping $T$ that maps an integral kernel to the corresponding integral operator. 
On the other hand, as noted in Remark~\ref{prior}, 
we have the isometric $*$-isomorphism 
$$\Psi\colon L^1(G,\RUCb(G,\CC);\alpha)\to\Ac_{G,\CC},\quad f\mapsto T_{R(f)}$$
which is given explicitly by 
$$((\Psi f)\phi)(x)=(T_{R(f)}\phi)(x)
=\int\limits_G (R(f))(x,y)\phi(y)\de y
=\int\limits_G f(xy^{-1},x)\phi(y)\de y $$
for all $f\in L^1(G,\RUCb(G,\CC);\alpha)$ and $\phi\in L^2(G)$, 
and this shows that $\Psi$ agrees with the $*$-homomorphism 
$\Lambda$ from \cite[subsect. 6.1]{Ge11}. 
Therefore the norm closure of $\Ran\Psi$ in $\Bc(L^2(G))$ is 
the reduced crossed product $\RUCb(G)\rtimes_r G$ 
(\cite[Th. 6.2]{Ge11}). 
Hence we have the continuous inclusion $\Ac_{G,\CC}\hookrightarrow\Ec(G)$ 
and the norm closure in $\Bc(L^2(G))$ of the unital $*$-algebra $\CC\1+\Ac_{G,\CC}$ is equal to $\Ec_1(G)$. 
It also follows by Theorem~\ref{main} that $\CC\1+\Ac_{G,\CC}$ is inverse-closed in $\Bc(L^2(G))$, 
hence it is also inverse closed in $\Ec_1(G)$, 
and this completes the proof of Assertion~\eqref{final_item3} as well.  

For proving Assertion~\eqref{final_item2} pick any $\phi,\psi\in\Cc(G,\CC)$ with compact supports. 
Then $(R(\phi\otimes\psi))(x,y)=\phi(xy^{-1})\psi(x)$ for all $x,y\in G$, 
hence the integral operator $T_{R(\phi\otimes\psi)}$ belongs to $\Ac_{G,\CC}\cap\Kc(L^2(G))$. 
Moreover, the integral operators of this type span a dense subspace of $\Kc(L^2(G))$ 
(see for instance \cite[Lemma 5.2.8]{Va85}). 

As regards Assertion~\eqref{final_item4}, 
it follows by \cite[Prop. 6.4--6.5]{Ge11} that $\Ec(G)$ is invariant under $\Ad\,\rho(\cdot)$, 
while the invariance of $\Ec(G)$ under $\Ad\,\lambda(\cdot)$ is a straightforward consequence of the definitions. 
For proving the assertions on $\Ac_{G,\CC}$ it is convenient to denote 
$$((\rho(a)\otimes\rho(a))K)(x,y)=K(xa,ya)\text{ and }((\lambda(a)\otimes\lambda(a))K)(x,y)=K(a^{-1}x,a^{-1}y)$$ 
for all $a,x,y\in G$ and $K\colon G\times G\to\CC$. 
Then it is easily checked that 
for every integral kernel $K\in\Kern(G,\CC)$, the corresponding integral operator $T_K\in\Bc(L^2(G))$ satisfies 
for arbitrary $a\in G$,  
\begin{equation}\label{final_proof_eq1}
 \rho(a)T_K\rho(a)^{-1}=T_{(\rho(a)\otimes\rho(a))K}
\text{ and }
\lambda(a)T_K\lambda(a)^{-1}=T_{(\lambda(a)\otimes\lambda(a))K}.
\end{equation}
Therefore, if $f\in L^1(G,\RUCb(G,\CC);\alpha)$, then for all $a,x,y\in G$ we have 
$$((\rho(a)\otimes\rho(a))R(f))(x,y)=(R(f))(xa,ya)=f(xy^{-1},xa)$$ 
(see Proposition~\ref{R}) and then 
$$((R^{-1}(\rho(a)\otimes\rho(a))R)(f))(x,y)=(\rho(a)\otimes\rho(a))R(f))(y,x^{-1}y)=f(x,ya).$$
Since the space $\RUCb(G,\CC)$ is invariant under right translations, 
it follows by the above formula that $(R^{-1}(\rho(a)\otimes\rho(a))R)(f)\in L^1(G,\RUCb(G,\CC);\alpha)$ 
and the mapping $f\mapsto  (R^{-1}(\rho(a)\otimes\rho(a))R)(f)$ is an isometry of $L^1(G,\RUCb(G,\CC);\alpha)$. 
Then by using the first equality in \eqref{final_proof_eq1} and the definition of $\Ac_{G,\CC}$ 
(see Theorem~\ref{main}), we obtain that $\Ac_{G,\CC}$ is invariant under $\Ad\,\rho(a)$, 
and the restriction of $\Ad\,\rho(a)$ to $\Ac_{G,\CC}$ is an isometric $*$-automorphism. 

Similarly, 
$$((\lambda(a)\otimes\lambda(a))R(f))(x,y)=(R(f))(a^{-1}x,a^{-1}y)=f(a^{-1}xy^{-1}a,a^{-1}x)$$ 
and then 
$$((R^{-1}(\lambda(a)\otimes\lambda(a))R)(f))(x,y)=(\lambda(a)\otimes\lambda(a))R(f))(y,x^{-1}y)=f(a^{-1}xa,a^{-1}y).$$
The space $\RUCb(G,\CC)$ is also invariant under left translations, 
hence by the above formula we have that $(R^{-1}(\lambda(a)\otimes\lambda(a))R)(f)\in L^1(G,\RUCb(G,\CC);\alpha)$ 
and the mapping $f\mapsto  (R^{-1}(\lambda(a)\otimes\lambda(a))R)(f)$ is an isometry of $L^1(G,\RUCb(G,\CC);\alpha)$. 
Then by using the second equality in \eqref{final_proof_eq1} and the definition of $\Ac_{G,\CC}$, 
we see that $\Ac_{G,\CC}$ is invariant under $\Ad\,\lambda(a)$, 
and the restriction of $\Ad\,\lambda(a)$ to $\Ac_{G,\CC}$ is an isometric $*$-automorphism. 

In addition, the action of $G$ by left translations on $\RUCb(G,\CC)$ is continuous (see Remark~\ref{syst}).  
By taking into account the above formulas, 
it follows that if 
$f=\phi\otimes\psi$ with $\phi\in L^1(G)$ and $\psi\in\RUCb(G)$, 
then we have 
\begin{equation}\label{final_proof_eq2}
\lim\limits_{a\to\1}\Vert (\Ad\,\lambda(a))T_{R(f)}-T_{R(f)}\Vert_{\Ac_{G,\CC}}=0.
\end{equation} 
Since the algebraic tensor product $L^1(G)\otimes\RUCb(G)$ is dense in the projective tensor product  
$L^1(G)\widehat{\otimes}\RUCb(G)=L^1(G,\RUCb(G,\CC);\alpha)$ 
and we already proved that the restriction of $\Ad\,\lambda(a)$ to $\Ac_{G,\CC}$ is an isometry for every $a\in G$, 
a standard approximation argument shows that \eqref{final_proof_eq2} holds true for any 
$f\in L^1(G,\RUCb(G,\CC),G;\alpha)$. 
We thus obtain that the mapping \eqref{final_eq1} 
from the statement is continuous, and this completes the proof. 
\end{proof}

The above theorem emphasizes the close relationship between the $C^*$-algebra $\Ec_1(G)$ 
and the Banach algebra $\CC\1+\Ac_{G,\CC}$, 
which not only continuously embeds as a dense inverse-closed subalgebra, 
but also shares with $\Ec_1(G)$ a couple of natural symmetry groups. 
It would be quite interesting to determine the whole group of automorphisms of $\Ec(G)$ 
that leave invariant the subalgebra $\Ac_{G,\CC}$.

\begin{corollary}\label{last}
Let $G$ is any finite-dimensional Lie group which is unimodular, amenable, and 
such that its discrete undelying group $G_d$ is rigidly symmetric. If we define define 
$$\Ac_{G,\CC}^\infty=\{T\in\Ac_{G,\CC}\mid (\Ad\,\lambda(\cdot))T\in\Ci(G,\Ac_{G,\CC})\}, $$
then $\CC\1+\Ac_{G,\CC}^\infty$ is a dense inverse-closed $*$-subalgebra of the unital $C^*$-algebra~$\Ec_1(G)$. 
Moreover, $\CC\1+\Ac_{G,\CC}^\infty$ has the structure of a Fr\'echet algebra with continuous inversion and  
is invariant under the action \eqref{final_eq1}. 
The corresponding action of $G$ on $\Ac_{G,\CC}^\infty$ is smooth 
and gives rise to a natural representation of the Lie algebra $\gg$ of $G$ 
by derivations of $\Ac_{G,\CC}^\infty$. 
\end{corollary}

\begin{proof} 
The continuity of the group action \eqref{final_eq1}  
actually shows that we have a continuous isometric representation of $G$ on $\Ac_{G,\CC}$.  
Then we obtain the dense subalgebra $\Ac_{G,\CC}^\infty$ of $\Ac_{G,\CC}$ consisting of differentiable vectors 
for that representation, just as in \cite[Cor. 3.1(ii)]{BB13}. 
Since $\Ac_{G,\CC}$ is a Banach algebra, it is easily checked that $\CC\1+\Ac_{G,\CC}^\infty$ is an 
inverse-closed subalgebra of $\CC\1+\Ac_{G,\CC}$, 
and then by Theorem~\ref{final}\eqref{final_item3} 
we obtain that $\CC\1+\Ac_{G,\CC}^\infty$ is also an  
inverse-closed subalgebra of $\Ec(G)$. 
Finally, the assertions on the topology of $\Ac_{G,\CC}^\infty$ follow by \cite[Th. 6.2]{Ne10}, 
which actually holds true for isometric Lie group actions on Banach algebras. 
\end{proof}

To conclude, we mention that some motivation for the above Corollary~\ref{last} 
can be found in the recent results of \cite{BG13} and \cite{BB13}. 
For instance, the framework in \cite{BG13} (see also \cite{BG08}) is provided by 
a so-called Lie $C^*$-system, which means a pair $(X,\Ac)$ consisting of 
some unital $C^*$-algebra $\Ac$ which admits faithful irreducible representations  
and is endowed with an injective homomorphism of Lie algebras 
$\delta\colon X\to \Der\Ac_0$, where $\Ac_0$ is some dense unital $*$-subalgebra of~$\Ac$. 
As mentioned in \cite{BG13}, this framework covers quantum physics, 
where algebras of observables are constructed in terms of some distinguished representation, 
as for instance the Fock representation. 

In the case when the Lie algebra $X$ is abelian, one studied  
certain pseudo-resolvents associated to $\delta$, 
which are families of elements of the dense $*$-subalgebra $\Ac_0$ 
that behave as resolvents of self-adjoint operators affiliated with $\Ac$ 
in some sense and which satisfy suitable commutation relations in terms of~$\delta$. 
One of the problems suggested in \cite{BG13} is that of extending their results to non-abelian Lie algebras. 

On the other hand,  it is clear that in the setting of our Corollary~\ref{last} 
(see also Theorem~\ref{final}) 
the derivative of the group homomorphism 
$\Ad\lambda(\cdot)\colon G\to\Aut\Ec_1(G)$ 
is a representation $\delta\colon\gg\to\Der\Ec_1(G)$ of the Lie algebra $\gg$  
by unbounded derivations of the $C^*$-algebra $\Ec_1(G)$.   
We thus obtain a Lie $C^*$-system in the above sense, 
for which the role of the dense $*$-subalgebra $\Ac_0$ (the common invariant domain 
of the unbounded derivations in the range of~$\delta$) 
is played by the Fr\'echet algebra $\Ac_{G,\CC}^\infty$ from the above Corollary~\ref{last}. 
Therefore it is natural to wonder what is the bearing of our specific examples of Lie $C^*$-systems 
on the problem raised in~\cite{BG13}. 

\appendix
\section{}
This appendix is dedicated to the proof of Corollary~\ref{leptin2}.
We need to introduce some extra notation. 

To any $C^*$-dynamical system $(G,\Ac,\alpha)$ there corresponds the $C^*$-dynamical system $(G_d,\Ac,\alpha)$. 
If $\Bg:=L^1(G,\Ac;\alpha)$ and $\Bg_d:=\ell^1(G_d,\Ac;\alpha)$ are their covariance algebras,  
then every simple $\Bg$-module $E$ can be made into a nondegenerate contractive Banach $\Bg$-module (see \cite[page 191]{Po92}). 

\begin{definition}\label{II2}
\normalfont 
Every nondegenerate contractive Banach $\Bg$-module $E$ is a nondegenerate contractive Banach $\Mc(\Bg)$-module in a canonical way by \cite[Th. 4.5(1)]{DJW09}, where $\Mc(\Bg)$ is the multiplier algebra of~$\Bg$. 
Then for $\Bg=L^1(G,\Ac;\alpha)$  
one obtains by \cite[Prop. 2.1]{LP91} a covariant representation of $(G,\Ac,\alpha)$ on $E$, 
hence a covariant representation of $(G_d,\Ac,\alpha)$ on $E$, 
and eventually a structure of $\Bg_d$-module on $E$, 
called the \emph{discretization} of the nondegenerate contractive Banach $\Bg$-module $E$. 
\end{definition}

\begin{lemma}\label{II3}
Let $\Xc$ and $\Yc$ be any Banach spaces. 
If $T\colon \Xc\to\Yc$ is a bounded linear operator satisfying the condition 
$$(\forall y\in\Yc)(\forall\epsilon>0)(\exists x\in\Xc)\quad 
\max\{\Vert x\Vert-\Vert y\Vert, \Vert Tx-y\Vert\}<\epsilon$$
then $T$ is surjective. 
\end{lemma}

\begin{proof}
See the final paragraph of the proof of \cite[Th. 2]{Po92}. 
\end{proof}

\begin{lemma}\label{II4}
For every simple $\Bg$-module $E$, its discretization is a simple $\Bg_d$-module. 
\end{lemma}

\begin{proof}
Denote the $\Bg$-module structure of $E$ by 
$\Bc\times E\to E$, $(f,y)\mapsto \rho(f)y$
and the $\Bg_d$-module structure of $E$ by
$\Bc\times E\to E$, $(\varphi,y)\mapsto \rho_d(\varphi)y$. 
To prove that the $\Bg_d$-module $E$ is simple, one checks that  
for arbitrary $y_0\in E$ 
the operator $\Bg_d\to E$, $\varphi\mapsto \rho_d(\varphi)y_0$ is surjective. 
To this end we use Lemma~\ref{II3}. 
The hypothesis of that lemma is satisfied with the norm on $E$ that comes from the fact that  
the simple $\Bg$-module $E$ can be made into a nondegenerate contractive Banach $\Bg$-module as noted above. 
The method of proof of \cite[Th. 2]{Po92} carries over to the present setting. 
\end{proof}

\begin{proof}[Proof of Corollary~\ref{leptin2}]
By using Proposition~\ref{leptin1} for the discrete group $G_d$,  
it follows that the covariance algebra $\ell^1(G_d,\Ac;\alpha)$ 
is isometrically $*$-isomorphic to a closed involutive subalgebra 
of $\ell^1(G,\bar\Ac;\bar\alpha)\simeq \ell^1(G)\hotimes\bar\Ac$ for a suitable $C^*$-algebra~$\bar\Ac$. 
On the other hand, the involutive Banach algebra $\ell^1(G_d)\hotimes\bar\Ac$ 
is symmetric since the discrete group $G_d$ is rigidly symmetric. 
Now, since any closed involutive subalgebra of a symmetric Banach algebra is in turn symmetric 
(see for instance \cite[Prop. 7.10]{Bi10}), 
it follows that $\ell^1(G_d,\Ac;\alpha)$ is a symmetric Banach algebra. 

To prove that also $L^1(G,\Ac;\alpha)$ is symmetric, 
let $E$ be any simple $L^1(G,\Ac;\alpha)$-module. 
The discretization of $E$ is a simple $\ell^1(G_d,\Ac;\alpha)$-module by Lemma~\ref{II4}. 
We proved above that $\ell^1(G_d,\Ac;\alpha)$ is a symmetric Banach algebra, 
so by \cite[(1)]{Le76} there exists a continuous positive sesquilinear form 
$(\cdot\mid\cdot)$ on $E$ with 
$(\rho_d(\phi)x\mid y)=(x\mid\rho_d(\phi^*)y)$ for all $x,y\in E$ and $\phi\in\ell^1(G_d,\Ac;\alpha)$. 
By using the method of proof of \cite[Cor. 6]{Po92} 
one can then prove that 
$$(\rho(f)x\mid y)=(x\mid\rho(f^*)y) \text{ for all }x,y\in E\text{ and }f\in L^1(G,\Ac;\alpha).$$
Thus $L^1(G,\Ac;\alpha)$ is a symmetric algebra by \cite[(1)]{Le76}, 
and this concludes the proof. 
\end{proof}

\subsection*{Acknowledgment} 
We are grateful to the two Referees for their numerous remarks and generous suggestions 
(in particular Remark~\ref{referee2}), 
which greatly helped us to improve the presentation and to correct several inaccuracies. 

This research has been partially supported by the Grant
of the Romanian National Authority for Scientific Research, CNCS-UEFISCDI,
project number PN-II-ID-PCE-2011-3-0131. 
The second-named author also acknowledges partial support from the Project MTM2010-16679, DGI-FEDER, of the MCYT, Spain.


\begin{thebibliography}{1000000}

\bibitem[Ba97]{Ba97}
\textsc{A.G.~Baskakov}, 
Asymptotic estimates for elements of matrices of inverse operators, and harmonic analysis. (Russian) 
\textit{Sibirsk. Mat. Zh.} \textbf{38} (1997), no.~1, 14--28, i; 
translation in \textit{Siberian Math. J.} \textbf{38} (1997), no.~1, 10--22. 

\bibitem[BB12]{BB12}
\textsc{I.~Belti\c t\u a, D.~Belti\c t\u a}, 
Algebras of symbols associated with the Weyl calculus for Lie group representations. 
\textit{Monatsh. Math.} \textbf{167} (2012), no.~1, 13--33.

\bibitem[BB13]{BB13}
\textsc{I.~Belti\c t\u a, D.~Belti\c t\u a}, 
On the differentiable vectors for contragredient representations. 
\textit{C. R. Math. Acad. Sci. Paris} \textbf{351} (2013), no. 13--14, 513--516.

\bibitem[Bi10]{Bi10}
\textsc{H.~Biller}, 
Continuous inverse algebras with involution. 
\textit{Forum Math.} \textbf{22} (2010), no.~6, 1033-1059. 

\bibitem[BC08]{BC08a}
\textsc{O.~Blasco, J.M.~Calabuig}, 
Vector-valued functions integrable with respect to bilinear maps. 
\textit{Taiwanese J. Math.} \textbf{12} (2008), no.~9, 2387--2403. 

\bibitem[BG08]{BG08}
\textsc{D.~Buchholz, H.~Grundling}, 
The resolvent algebra: a new approach to canonical quantum systems. 
\textit{J. Funct. Anal.} \textbf{254} (2008), no.~11, 2725--2779. 

\bibitem[BG13]{BG13}
\textsc{D.~Buchholz, H.~Grundling}, 
Lie algebras of derivations and resolvent algebras. 
\textit{Comm. Math. Phys.} \textbf{320} (2013), no.~2, 455--467.

\bibitem[CR66]{CR66}
\textsc{W.W.~Comfort, K.A.~Ross}, 
Pseudocompactness and uniform continuity in topological groups. 
\textit{Pacific J. Math.} \textbf{16} (1966), 483--496.

\bibitem[DG04]{DG04}
\textsc{M.~Damak, V.~Georgescu}, 
Self-adjoint operators affiliated to $C^*$-algebras. 
\textit{Rev. Math. Phys.} \textbf{16} (2004), no.~2, 257--280.


\bibitem[DJW09]{DJW09} 
S.~Dirksen, M.~de Jeu, M.~Wortel, 
Extending representations of normed algebras in Banach spaces. 
In: M. de Jeu, S. Silvestrov, C. Skau, J. Tomiyama (eds.), 
\textit{Operator structures and dynamical systems}. 
Contemp. Math., 503, Amer. Math. Soc., Providence, RI, 2009, pp.~53--72. 

\bibitem[El58]{El58}
\textsc{H.W.~Ellis}, 
A note on Banach function spaces. 
\textit{Proc. Amer. Math. Soc.} \textbf{9} (1958), 75--81. 

\bibitem[EH53]{EH53}
\textsc{H.W.~Ellis, I.~Halperin}, 
Function spaces determined by a levelling length function. 
\textit{Canadian J. Math.} \textbf{5} (1953), 576--592.

\bibitem[FS10]{FS10}
\textsc{B.~Farrell, T.~Strohmer}, 
Inverse-closedness of a Banach algebra of integral operators on the Heisenberg group. 
\textit{J. Operator Theory} \textbf{64} (2010), no.~1, 189--205. 

\bibitem[Fe83]{Fe83}
\textsc{H.G.~Feichtinger}, 
Banach convolution algebras of Wiener type. 
In: \textit{Functions, Series, Operators} (Budapest, 1980), Vol. I-- II, 
Colloq. Math. Soc. J\'anos Bolyai, 35, North-Holland, Amsterdam, 1983, pp.~509--524.

\bibitem[FD88]{FD88}
\textsc{J.M.G.~Fell, R.S.~Doran}, 
\textit{Representations of $^*$-Algebras, Locally Compact Groups, and Banach $\sp *$-Algebraic Bundles}. Vol. 1.  
Pure and Applied Mathematics, 125. Academic Press, Inc., Boston, MA, 1988. 

\bibitem[FGL08]{FGL08}
\textsc{G.~Fendler, K.~Gr\"ochenig, M.~Leinert}, 
Convolution-dominated operators on discrete groups. 
\textit{Integral Equations Operator Theory} \textbf{61} (2008), no.~4, 493--509.

\bibitem[Ge11]{Ge11}
\textsc{V.~Georgescu}, 
On the structure of the essential spectrum of elliptic operators on metric spaces. 
\textit{J. Funct. Anal.} \textbf{260} (2011), no.~6, 1734--1765.

\bibitem[GI02]{GI02}
\textsc{V.~Georgescu, A.~Iftimovici}, 
Crossed products of $C^*$-algebras and spectral analysis of quantum Hamiltonians. 
\textit{Comm. Math. Phys. 228} (2002), no.~3, 519--560. 

\bibitem[GMS91]{GMS91}
\textsc{C.~G\'erard, A.~Martinez, J.~Sj\"ostrand},
A mathematical approach to the effective Hamiltonian in perturbed periodic problems.
\textit{Comm. Math. Phys.} {\bf 142} (1991), no.~2, 217--244

\bibitem[GW04]{GW04}
\textsc{M.~Girardi, L.~Weis}, 
Integral operators with operator-valued kernels. 
\textit{J. Math. Anal. Appl.} \textbf{290} (2004), no.~1, 190--212.

\bibitem[Gr81]{Gr81}
\textsc{M.~Gromov}, 
Groups of polynomial growth and expanding maps. 
\textit{Inst. Hautes \'Etudes Sci. Publ. Math.} No. {\bf 53} (1981), 53--73. 

\bibitem[Ha53]{Ha53}
\textsc{I.~Halperin}, 
Function spaces. 
\textit{Canadian J. Math.} \textbf{5} (1953), 273--288.

\bibitem[Hu72]{Hu72}
\textsc{A.~Hulanicki}, 
On the spectrum of convolution operators on groups with polynomial growth. 
\textit{Invent. Math.} \textbf{17} (1972), 135--142.

\bibitem[Ki62]{Ki62}
\textsc{J.M.~Kister}, 
Uniform continuity and compactness in topological groups. 
\textit{Proc. Amer. Math. Soc.} \textbf{13} (1962), 37--40. 

\bibitem[Ku99]{Ku99}
\textsc{V.G.~Kurbatov}, 
\textit{Functional-Differential Operators and Equations}. 
Mathematics and its Applications, 473. 
Kluwer Academic Publishers, Dordrecht, 1999. 

\bibitem[Ku01]{Ku01}
\textsc{V.G.~Kurbatov}, 
Some algebras of operators majorized by a convolution.  
\textit{Funct. Differ. Equ.} \textbf{8} (2001), no.~3-4, 323--333.

\bibitem[LP91]{LP91}
A.T.~Lau, A.L.T.~Paterson, 
Amenability for twisted covariance algebras and group $C^*$-algebras. 
\textit{J. Funct. Anal.} \textbf{100} (1991), no.~1, 59--86. 


\bibitem[Le68]{Le68}
\textsc{H.~Leptin}, 
Darstellungen verallgemeinerter $L^1$-Algebren. 
\textit{Invent. Math.} \textbf{5} (1968), 192--215.

\bibitem[Le74]{Le74}
\textsc{H.~Leptin}, 
On symmetry of some Banach algebras. 
\textit{Pacific J. Math.} \textbf{53} (1974), 203--206. 

\bibitem[Le76]{Le76}
H.~Leptin, 
Symmetrie in Banachschen Algebren. 
\textit{Arch. Math. (Basel)} \textbf{27} (1976), no.~4, 394--400.


\bibitem[LP79]{LP79}
\textsc{H.~Leptin, D.~Poguntke}, 
Symmetry and nonsymmetry for locally compact groups. 
\textit{J. Funct. Anal.} \textbf{33} (1979), no.~2, 119--134.

\bibitem[Ne10]{Ne10}
\textsc{K.-H.~Neeb}, 
On differentiable vectors for representations of infinite dimensional Lie groups. 
\textit{J. Funct. Anal.} \textbf{259} (2010), no.~11, 2814--2855. 

\bibitem[Pl01]{Pl01}
\textsc{Th.W.~Palmer}, 
\textit{Banach Algebras and the General Theory of $*$-algebras}. Vol.~2.  
Encyclopedia of Mathematics and its Applications, 79. Cambridge University Press, Cambridge, 2001. 

\bibitem[Pt88]{Pa88}
\textsc{A.L.T.~Paterson}, 
\textit{Amenability}. 
Mathematical Surveys and Monographs, 29. 
American Mathematical Society, Providence, RI, 1988.

\bibitem[Pe79]{Pe79}
\textsc{G.K.~Pedersen}, 
\textit{$C^*$-algebras and Their Automorphism Groups}. 
London Mathematical Society Monographs, 14. Academic Press, Inc., 
London-New York, 1979.

\bibitem[Po92]{Po92}
\textsc{D.~Poguntke}, 
Rigidly symmetric $L^1$-group algebras. 
\textit{Sem. Sophus Lie} \textbf{2} (1992), no.~2, 189--197.

\bibitem[RR06]{RR06}
\textsc{V.S.~Rabinovich, S.~Roch}, 
The essential spectrum of Schr�dinger operators on lattices. 
\textit{J. Phys. A} \textbf{39} (2006), no.~26, 8377--8394.

\bibitem[Ta75]{Ta75}
\textsc{H.~Takai}, 
On a duality for crossed products of $C^*$-algebras. 
\textit{J. Functional Analysis} \textbf{19} (1975), 25--39. 

\bibitem[Va85]{Va85}
\textsc{J.-M.~Vallin}, 
$C^*$-alg\`ebres de Hopf et $C^*$-alg\`ebres de Kac. 
\textit{Proc. London Math. Soc. (3)} \textbf{50} (1985), no.~1, 131--174.

\bibitem[Wi07]{Wi07}
\textsc{D.P.~Williams}, 
\textit{Crossed Products of $C^*$-algebras}. 
Mathematical Surveys and Monographs, 134. 
American Mathematical Society, Providence, RI, 2007. 

\end{thebibliography}
\end{document}